\documentclass[11pt,leqno]{article}
\usepackage{amsmath, amscd, amsthm, amssymb, graphics, xypic, mathrsfs, setspace, fancyhdr, bm, pdfsync, enumitem, mathptmx}
\usepackage[usenames, dvipsnames, svgnames, table]{xcolor}
\usepackage[letterpaper,top=1.05in, bottom=1.05in, left=1.05in, right=1.05in]{geometry}
\usepackage[colorlinks=true,pagebackref=true]{hyperref}
\hypersetup{backref}

\newcommand{\hocofib}{\operatorname{hocofib}}
\newcommand{\Spec}{\operatorname{Spec}}

\newcommand{\isomto}{{\stackrel{\sim}{\;\longrightarrow\;}}}
\newcommand{\isomt}{{\stackrel{{\scriptscriptstyle{\sim}}}{\;\rightarrow\;}}}

\newcommand{\sma}{{\scriptstyle{\wedge}\,}}

\newcommand{\op}[1]{\operatorname{#1}}

\renewcommand{\hom}{\operatorname{Hom}}

\newcommand{\cplx}{{\mathbb C}}

\newcommand{\Z}{{\mathbb Z}}

\newcommand{\D}{{\mathrm D}}
\newcommand{\Proj}{{\operatorname{Proj}}}
\newcommand{\aone}{{\mathbb A}^1}
\newcommand{\pone}{{\mathbb P}^1}

\newcommand{\ga}{{{\mathbb G}_{a}}}
\newcommand{\gm}[1]{{{\mathbb G}_{m}^{#1}}}

\newcommand{\ho}[2][]{\mathrm{H}_{#1}({#2})}

\newcommand{\Nis}{{\operatorname{Nis}}}
\newcommand{\Zar}{\operatorname{Zar}} 

\renewcommand{\deg}{\operatorname{deg}}

\newcommand{\aff}{\operatorname{aff}}

\newcommand{\Sm}{\mathrm{Sm}}

\newcommand{\Spc}{\mathrm{Spc}}

\newcommand{\Th}{{\operatorname{Th}}}

\renewcommand{\setminus}{\smallsetminus}

\newcommand{\Addresses}{{
\bigskip
\footnotesize

A.~Asok, Department of Mathematics, University of Southern California, 3620 S.~Vermont Ave., Los Angeles, CA 90089-2532, United States; E-mail address: asok@usc.edu
\medskip

A.~Dubouloz, IMB UMR5584, CNRS, Universit{\'e} de Bourgogne Franche-Comt{\'e}, Dijon, France; E-mail address: Adrien.Dubouloz@u-bourgogne.fr
\medskip

P.~A.~{\O}stv{\ae}r, 
Department of Mathematics F. Enriques, University of Milan, Via Cesare Saldini 50, 20133 Milan, Italy; E-mail address: paul.oestvaer@unimi.it \&
Department of Mathematics, University of Oslo, P.~O.~Box 1053 Blindern, 0316 Oslo, Norway; E-mail address: paularne@math.uio.no
}}

\newcounter{intro}
\theoremstyle{plain}
\newtheorem{thm}{Theorem}[subsection]

\newtheorem{lem}[thm]{Lemma}
\newtheorem{cor}[thm]{Corollary}
\newtheorem{prop}[thm]{Proposition}
\newtheorem*{claim*}{Claim} 

\newtheorem{question}[thm]{Question}

\newtheorem*{thm*}{Theorem}
\newtheorem*{problem*}{Problem}

\newtheorem{thmintro}{Theorem}

\newtheorem{questionintro}[thmintro]{Question}

\theoremstyle{definition}
\newtheorem{defn}[thm]{Definition}
\newtheorem{construction}[thm]{Construction}

\theoremstyle{remark}
\newtheorem{rem}[thm]{Remark}
\newtheorem{remintro}[thmintro]{Remark}

\newtheorem{ex}[thm]{Example}

\newtheorem{entry}[thm]{}

\numberwithin{equation}{subsection}

\begin{document}
\pagestyle{fancy}
\renewcommand{\sectionmark}[1]{\markright{\thesection\ #1}}
\fancyhead{}
\fancyhead[LO,R]{\bfseries\footnotesize\thepage}
\fancyhead[LE]{\bfseries\footnotesize\rightmark}
\fancyhead[RO]{\bfseries\footnotesize\rightmark}
\chead[]{}
\cfoot[]{}
\setlength{\headheight}{1cm}

\author{
Aravind Asok\thanks{Aravind Asok was partially supported by National Science Foundation Awards DMS-1802060 and DMS-2101898} 
\and 
Adrien Dubouloz\thanks{Adrien Dubouloz was partially supported by Project ISITE-BFC ANR-15-IDEX-0008 "Motivic Invariants of Algebraic Varieties" and  ANR Project PRC ``HQDIAG" ANR-21-CE40-0015-02}
\and 
Paul Arne {\O}stv{\ae}r\thanks{Paul Arne {\O}stv{\ae}r was partially supported by RCN Project no. 312472 ``Equations in Motivic Homotopy Theory"}
}

\title{{\bf Geometric models for algebraic suspensions}}
\date{}


\maketitle

\begin{abstract}
We analyze the question of which motivic homotopy types admit smooth schemes as representatives.  We show that given a pointed smooth affine scheme $X$ and an embedding into affine space, the affine deformation space of the embedding gives a model for the ${\mathbb P}^1$ suspension of $X$; we also analyze a host of variations on this observation.
Our approach yields many examples of ${\mathbb A}^1$-$(n-1)$-connected smooth affine $2n$-folds and strictly quasi-affine ${\mathbb A}^1$-contractible smooth schemes. 
\end{abstract}

\section*{Introduction}
A very basic question in topology is: which homotopy types admit (smooth) manifold representatives?  When a given homotopy type admits a manifold representative, one can then ask: can the homeomorphism (resp. diffeomorphism) types be enumerated?  Broadly speaking, this note concerns algebro-geometric variants of such questions: we replace manifolds by smooth algebraic varieties (frequently affine) and the ordinary homotopy category by the Morel--Voevodsky $\aone$-homotopy category \cite{MV}.  To guide the discussion, consider the following:
\begin{questionintro}
\begin{enumerate}[noitemsep,topsep=1pt]
\item Which motivic homotopy types admit representatives that are smooth schemes?  
\item If a motivic homotopy type admits at least one smooth representative, can one distinguish the isomorphism classes of distinct representatives?  
\end{enumerate}
\end{questionintro}

An implicit impediment to making the above questions precise is the issue of specifying motivic homotopy types.  The simplest motivic homotopy type, namely that of a point, admits many non-isomorphic representatives. We refer the reader to \cite{AsokOstvaer} for a survey of results in this direction.  For this reason, instead of trying to specify motivic homotopy types, we will focus on describing how such types can change.  Here is a sample result.

\begin{thmintro}[See Theorem~\ref{thm:main}]
	\label{thmintro:onedegeneratefiber}
	Assume $B$ is a scheme, $\mathfrak{X}$ is a $B$-scheme, and $\pi: \mathfrak{X} \to {\mathbb A}^n_B$ is a smooth morphism admitting a section $s$.  If, setting $U := \pi^{-1}({\mathbb A}^n \setminus 0)$, the morphism $\pi|_U: \mathfrak{X}|_U \to U$ is an $\aone$-weak equivalence, then there is an induced (pointed) $\aone$-weak equivalence
	\[
	{\pone}^{\sma n} \wedge (\mathfrak{X}_0,s(0)) \sim (\mathfrak{X},s(0)).
	\]
\end{thmintro}

\begin{remintro}
	\label{remintro:genericaoneweakequivalence}
	Theorem~\ref{thmintro:onedegeneratefiber} is deduced from the more general result Theorem~\ref{thm:mainmultiple}.  Rather than stating this more general result here, let us explain the general principle at work.  A smooth morphism $\pi: \mathfrak{X} \to Y$ will be called a {\em generic $\aone$-weak equivalence} if there exists a dense open subscheme $U \subset Y$ such that $\pi|_U$ is an $\aone$-weak equivalence.  In that setting, the motivic homotopy type of $\mathfrak{X}$ can be obtained by gluing (as in \cite[Theorem 2.21]{MV}) the motivic homotopy type of $U$ with the motivic homotopy type of $\mathfrak{X}|_{Y \setminus U}$.  Under suitable additional hypotheses, we can control the $\aone$-homotopy type of the gluing.
\end{remintro}

Our next order of business is to create an ample supply of generic $\aone$-weak equivalences to which our techniques apply.  One source of such morphisms is provided by the deformation to the normal cone construction a la Fulton-MacPherson: the deformation space of an embedding of a smooth affine variety $k$-variety in ${\mathbb A}^n_k$ will yield a family to which the preceding result may be applied.  The following result is a straightforward consequence of the results above, combined with the Jouanolou--Thomason homotopy lemma.

\begin{thmintro}[See Corollary~\ref{cor:modelsofiteratedsuspension}]
	\label{thmintro:main}
	Suppose $B$ is a quasi-compact and quasi-separated (qcqs) base scheme.  Assume $(X,x)$ is a finitely presented pointed smooth $B$-scheme.
	\begin{enumerate}[noitemsep,topsep=1pt]
		\item If the structure map $X \to B$ is furthermore affine, then for any integer $i \geq 0$, the motivic space $S^i \sma \gm{\sma i} \sma X$ has the $\aone$-homotopy type of a pointed $B$-scheme that is smooth and affine over $B$.
		\item If $B$ is affine and regular (e.g., the spectrum of $\Z$ or a field), then the same statement holds for any $B$-scheme $X$ that is smooth and has affine diagonal over $B$ (thus, any smooth separated $B$-scheme satisfies the hypotheses).
	\end{enumerate}
\end{thmintro}

\begin{remintro}
Theorem~\ref{thmintro:main} can be viewed as a refinement and extension of the technique conceived in \cite{AsokDoranFasel}.  In that paper, an inductive argument was used to show that the smooth affine quadric hypersurface $\op{Q}_{2n}$ defined by the equation $\sum_{i=1}^n x_iy_i = z(1-z)$ has the motivic homotopy type of the motivic sphere $S^i \sma \gm{\sma i}$ \cite[Theorem 2]{AsokDoranFasel}.  In retrospect, this smooth affine quadric is itself an iterated deformation space.  
\end{remintro}

Construction~\ref{construction:parameterizeddeformations} gives a generalization of the deformation to the normal cone construction.  Loosely speaking, this construction lets us specify a degeneration locus $Z$ in a suitable parameter space $W$ and then construct a morphism whose degeneration locus is precisely $Z$ with control over the fibers over $Z$.  Proposition~\ref{prop:smoothnessofparameterizeddeformationspaces} explains the relevant technical hypotheses that we will use to construct generic $\aone$-weak equivalences.  One reason for making Construction~\ref{construction:parameterizeddeformations} in such generality is that we obtain a plethora of smooth scheme models of a given $\aone$-homotopy type.  This flexibility allows us to encompass a large class of examples in affine algebraic geometry.  Indeed, the problem of distinguishing isomorphism classes of smooth affine schemes within a given $\aone$-homotopy type contains many classical problems in affine algebraic geometry and the examples we construct. 

The remainder of this note is then devoted to constructing several classes of examples.  First, we generalize the results of \cite{AsokDoranFasel} in two ways: a) we construct new examples of highly $\aone$-connected hypersurfaces by producing and analyzing a variation on Danielewski's construction; for example, we construct many examples of $\aone$-$(n-1)$-connected smooth affine $2n$-folds, b) we build new examples of strictly quasi-affine $\aone$-contractible smooth schemes.  The problem of classifying $\aone$-$(n-1)$-connected smooth affine $2n$-folds is reminiscent of (a non-compact version of) that studied for manifolds by C.T.C.\ Wall in \cite{Wall}.  Finally, we give examples to show that varieties that are not $\aone$-contractible can become so after $\pone$-suspension; these results provide evidence supporting \cite[Conjecture 14]{AsokOstvaer}.

\begin{thmintro}[See Theorem~\ref{thm:aonecontractiblehypersurfaces} and Remark~\ref{rem:topologicalcontractibility}]
	If $X$ is a topologically $\Z$-acyclic smooth complex surface that is not isomorphic to ${\mathbb A}^2_{\cplx}$,  then $X$ is not $\aone$-connected, but for every integer $N \geq 2$, the space ${\pone}^{\sma N} \wedge X$ has the $\aone$-homotopy type of a smooth affine $\cplx$-scheme and is $\aone$-contractible. 
\end{thmintro}

\subsubsection*{Acknowledgements}
We want to thank Marc Hoyois for helpful discussions about the proof of homotopy purity and Tom Bachmann for helpful discussions about conservativity.  {\O}stv{\ae}r gratefully acknowledges the hospitality and support of the Centre for Advanced Study at the Norwegian Academy of Science and Letters in Oslo, Norway, which funded and hosted the research project ``Motivic Geometry" during the 2020/21 academic year.

\section{Preliminaries and Notation}
Fix a base scheme $B$; if we do not explicitly mention otherwise in the body of the text, the reader should not assume any additional finiteness hypotheses are imposed upon $B$.  Next, write $\Sm_B$ for the category of schemes that are smooth over $B$.  For the convenience of referencing, we follow \cite[\href{https://stacks.math.columbia.edu/tag/01V5}{Tag 01V5}]{stacks-project} for our definition of smoothness as opposed to \cite[Definition 17.3.2]{EGAIV.4}; note that, either way, objects of $\Sm_B$ are locally of finite presentation over $B$.

\subsection{Simplicial presheaves and homotopy categories} 
Write $\Spc_B$ for the category of simplicial presheaves on $\Sm_B$.  We will typically use calligraphic letters to denote spaces.  Objects of $\Spc_B$ will be referred to as {\em $B$-spaces} or simply {\em spaces} if $B$ is clear from context.  The category $\Spc_B$ has $B$ as a final object, and we write $\Spc_{B,\ast}$ for the category of {\em pointed} spaces, i.e., spaces $\mathscr{X}$ provided with a morphism $B \to \mathscr{X}$ splitting the structure map.  The forgetful functor $\Spc_{B,\ast} \to \Spc_B$ has a left adjoint $(-)_+$ of ``adding a disjoint base-point"; this functor sends $\mathscr{X}$ to $\mathscr{X}_+ := \mathscr{X} \sqcup B$.  Categorical constructions like $\wedge$ and $\vee$ will all occur in $\Spc_{B,\ast}$, but sometimes we will explicitly indicate the base scheme $B$ (e.g., $\wedge_B$) for emphasis.

We view $\Sm_B$ as a site by equipping it with the Nisnevich topology (in general, this is the topology generated by the Nisnevich $cd$-squares).  We write $\ho{B}$ for the Morel--Voevodsky $\aone$-homotopy category: this category is constructed when $B$ is Noetherian of finite Krull dimension in \cite{MV}, but we follow \cite[Appendix C p. 3649]{HoyoisQuadraticRefinement} for the general case.  While $\ho{B}$ is described using the language of $\infty$-categories in \cite{HoyoisQuadraticRefinement}, one may also use model structures as described in \cite[\S 3.1, 4.1 and 5.1]{AHW}, \cite[\S 2]{DRO}.  In brief, we realize $\ho{B}$ as the left Bousfield localization of the injective $\Nis$-local model structure on $\Spc_B$ with respect to the $\aone$-weak equivalences.  Isomorphisms in $\ho{B}$ will be called $\aone$-weak equivalences, and we write
\[
[\mathscr{X},\mathscr{Y}]_{\aone} := \hom_{\ho{B}}(\mathscr{X},\mathscr{Y}),
\]
by analogy with homotopy classes of maps in topology.

\begin{remintro}
When $B$ is not Noetherian of finite Krull dimension, we caution the reader that the $\aone$-homotopy category as defined here differs from the naive extension of Morel--Voevodsky's definition; see \cite[Remark 3.1.4]{AHW} for more details.
\end{remintro}

\subsection{Cofiber sequences}
\label{s:homotopicalpreliminaries} 
We now collect for the reader's convenience some well-known useful facts about cofiber sequences in pointed model categories.
\label{ss:cofibersequences}
First, the motivic model structure we use is {\em left proper}, that is, pushouts of $\aone$-weak equivalences along cofibrations are again $\aone$-weak equivalences.  Left properness has the following useful consequence that we use repeatedly.

\begin{lem}
\label{lem:leftproper}
If $f: (\mathscr{X},x) \to (\mathscr{Y},y)$ is a morphism of pointed spaces and $(\mathscr{X},x)$ is $\aone$-contractible, then the canonical map
\[
(\mathscr{Y},y) \to \operatorname{hocofib}(f)
\]
is an $\aone$-weak equivalence.
\end{lem}

Next, we recall a general fact about the cofiber of a composite map (sometimes called the ``octahedral axiom" as it gives rise to the octahedral axiom in triangulated categories).

\begin{prop}
\label{prop:octahedralaxiom}
If $f_1: \mathscr{X}_2 \to \mathscr{X}_1$ and $f_0: \mathscr{X}_1 \to \mathscr{X}_0$ are pointed morphisms of spaces, then there is a cofiber sequence of the form:
\[
\operatorname{hocofib}(f_1) \longrightarrow \operatorname{hocofib}(f_0f_1) \longrightarrow \operatorname{hocofib}(f_0)
\]
\end{prop}

\begin{proof}
This is shown in \cite[Proposition 6.3.6]{Hovey}.
\end{proof}

The following fact has to do with smash products and base-points.  Assume $(\mathscr{X},x)$ and $(\mathscr{Y},y)$ are pointed spaces.  In this case, we may form the smash product $\mathscr{X} \wedge \mathscr{Y}$ in $\Spc_{B,\ast}$.  The base-point $x$ determines a canonical pointed map $x: S^0 \to \mathscr{X}_+$ sending the non-base-point of $S^0$ to the point $x$.  Smashing this morphism with $id_{{\mathscr Y}}$ then defines a morphism
\[
\mathscr{Y} = \mathscr{Y} \sma S^0 \stackrel{id \sma x}{\longrightarrow} \mathscr{Y}\sma  \mathscr{X}_+.
\]
This morphism is split by the morphism $\mathscr{X}_+ \to S^0$ that collapses $\mathscr{X}$ to the non-base-point.  It follows that the displayed morphism is a cofibration in the injective local model structure since it is a monomorphism.

\begin{prop}
\label{prop:basepointedsmash}
For any pointed spaces $(\mathscr{X},x)$ and $(\mathscr{Y},y)$, there is a canonical identification
\[
\operatorname{hocofib}(\mathscr{Y} \stackrel{id \sma x}{\longrightarrow} \mathscr{X}_+ \sma \mathscr{Y}) \cong \mathscr{X} \sma \mathscr{Y}.
\]
\end{prop}

\begin{proof}
Since the map in question is a cofibration, the homotopy cofiber coincides with the cofiber.  The identification of the cofiber is an exercise in unwinding the definition of the smash product (for more details, see \cite[Proposition 2.2.4]{AsokDoranFasel}).
\end{proof}

\section{Deformation to the normal cone revisited} 
\label{s:parametrized-def}
In this section, we study a variation on the ``deformation to the normal cone" construction, which appears in many places; see, e.g., \cite[\S IV.5]{FultonLang} or \cite[\S 2]{VerdierRR}.  Our eventual goal will be to use this kind of construction to produce generic $\aone$-weak equivalences $\pi: \mathfrak{X} \to W$ with good control over the closed subscheme $Z \subset W$ over which $\pi$ fails to be an $\aone$-weak equivalence.  

\subsection{Parameterized deformation spaces: construction}
\label{ss:constructionparamdefspaces}
Henceforth, we assume $B$ is an arbitrary scheme.  The deformation to the normal cone construction can be realized in terms of affine blow-ups.  The theory of affine blow-ups has been worked out in great generality in \cite[\S 2]{MRR} (see also \cite{KaZaVeryTransitive, DModAff} for special cases), but we will recall what we need here. 

\begin{entry}[Affine blow-ups]
	Assume $Z \subset D \subset X$ is a triple consisting of a scheme $X$, a closed subscheme $Z$ defined by a quasi-coherent sheaf of ideals $\mathscr{I}$ of ideals and a locally principal subscheme $D$ defined by a quasi-coherent sheaf of ideals $\mathscr{J}$ contained in  $\mathscr{I}$.  The affine blow-up $\op{Bl}_{Z}^D(X)$ is defined as follows: as usual, we write $\op{Bl}_{\mathscr{I}} \mathscr{O}_X$ for the (graded) blow-up algebra $\bigoplus_{n \geq 0} \mathscr{I}^n$ \cite[\href{https://stacks.math.columbia.edu/tag/052Q}{Tag 052Q}]{stacks-project}.  If we view the local generators of $\mathscr{J}$ as living in degree one, then $\op{Bl}_{Z}^D(X)$ is the complement of the variety $V_+(\mathscr{J})$ defined by the homogeneous ideal $\mathscr{J}$ in $\Proj \op{Bl}_\mathscr{I} \mathscr{O}_X$ \cite[Definition 2.1, Lemma 2.3]{MRR}.  With these preliminaries in mind, we introduce the following: 
\end{entry}

\begin{construction}
	\label{construction:parameterizeddeformations}
	Fix a base scheme $B$.  We assume given:
	\begin{enumerate}[noitemsep,topsep=1pt]
		\item A smooth $B$-scheme $W$ (``the parameter space"); 
		\item A locally finitely presented flat $B$-scheme $Z$ equipped with a closed immersion of $B$-schemes $i: Z \to W$ (``the degeneration locus");
		\item A smooth $B$-scheme $Y$; 
		\item A $B$-scheme $X$ together with a surjective smooth morphism $\psi: X \to Z$, and a closed immersion $f: X \to Y \times_B W$ such that the diagram
		\[
		\xymatrix{
			X \ar[r]^-{f}\ar[d]^-{\psi} & Y \times_B W \ar[d]^-{p_2}\\
			Z \ar[r]^-i & W
		}
		\]
		commutes; and
		\item A locally principal divisor $D$ on $Y\times_B W$ that contains $Y \times_B Z$.
	\end{enumerate}
	The {\em parameterized deformation space} 
	\[
	\pi: \op{D}(X,i,f) \longrightarrow W
	\]
	is the affine blow-up $\op{Bl}_{X}^D(Y \times_B W)$.
\end{construction}

The next example justifies the terminology {\em parameterized deformation space} for the output $\pi$ of Construction~\ref{construction:parameterizeddeformations}.

\begin{ex} 
	\label{ex:deformationspace}
	Take $W = {\mathbb A}^1_B$, fix a closed immersion of finitely presented smooth $B$-schemes $f: X \to Y$, let $Z = B$, take $i: Z \to \aone_B$ the zero section, and let $D$ be the Cartier divisor $Y\times_B Z\subset Y\times_B {\mathbb A}^1_B$.  In that case, if $f': X \to Y \times \aone_B$ is the base-change of $f$ along the projection map $Y \times \aone_B\to Y$, then the structure morphism $X \to B$ fits into a commutative square as in Construction~\ref{construction:parameterizeddeformations}, and $\op{D}(X,i,f)$ is the usual deformation space $\op{D}(X,Y)$ of the closed immersion $f$.
\end{ex}

\subsection{Parameterized deformation spaces: structural properties}
\label{ss:propertiesparamdefspaces}
This subsection aims to control the structure of parameterized deformation spaces. In particular, we want to exhibit hypotheses under which the morphism $\pi: \op{D}(X,i,f) \to W$ is smooth, and we would also like to control the structure of the exceptional locus.  Indeed, if we take $Y$ to be an $\aone$-contractible smooth $B$-scheme, then as long as $\pi$ is smooth, it is automatically a generic $\aone$-weak equivalence in the sense of the introduction.

In the setup of Construction~\ref{construction:parameterizeddeformations}, if $Z \subset W$ is an effective Cartier divisor, then the fiber product $Y \times_B Z \hookrightarrow Y \times_B W$ is an effective Cartier divisor as well; we will take $D = W \times_B Y$ in the sequel and we write $\mathscr{J}$ for the ideal sheaf of $Y \times_B Z$ in $Y \times_B W$.  In that case, the closed immersion $\psi: X \to Z$ necessarily factors through a closed immersion $f': X \hookrightarrow Y \times_B Z$.  We write $\mathscr{I}$ for the ideal sheaf of $X$ in $Y \times_B W$ and $\mathscr{J}_X$ for the restriction of $\mathscr{J}$ to $X$.  In that case, we write $\mathscr{C}_{f'} := \mathscr{I}/(\mathscr{I}^2 + \mathscr{J})$ for the conormal sheaf of $f'$ and $\mathscr{N}_{f'}$ for its dual, i.e., the normal sheaf of the embedding.

We want to speak about regular immersions, but doing so in the context when our base schemes have no finiteness hypotheses imposed upon them requires more work.  We refer the reader to \cite[\href{https://stacks.math.columbia.edu/tag/067M}{Tag 067M}]{stacks-project} and \cite[\href{https://stacks.math.columbia.edu/tag/0638}{Tag 0638}]{stacks-project} for discussions about regular immersion in this generality.  We freely use the fact that regular immersions always have locally free (co)normal sheaves (combine \cite[\href{https://stacks.math.columbia.edu/tag/067P}{Tag 067P}]{stacks-project} and \cite[\href{https://stacks.math.columbia.edu/tag/063M}{Tag 063M}]{stacks-project}).

Since smooth $B$-schemes are locally finitely presented by assumption, and since closed immersions are quasi-compact, it follows that a closed immersion of smooth $B$-schemes is automatically finitely presented.  Moreover, closed immersions of smooth $B$-schemes are automatically regular immersions \cite[\href{https://stacks.math.columbia.edu/tag/067U}{Tag 067U}]{stacks-project}.  Consequently, any closed immersion of smooth $B$-schemes has a well-defined normal sheaf $\mathscr{N}_{X/Y}$, which is a finite rank locally free $\mathscr{O}_X$-module.

\begin{prop}
	\label{prop:smoothnessofparameterizeddeformationspaces}
	Assume $(i: Z \to W, \psi: X \to Z, f: X \to Y \times_B W)$ are as in \textup{ Construction~\ref{construction:parameterizeddeformations}}, and furthermore suppose $Z$ is an effective Cartier divisor on $W$.
	\begin{enumerate}[noitemsep,topsep=1pt]
		\item The morphism $\pi^{-1}(W \setminus Z) \to W \setminus Z$ is isomorphic to the projection $Y \times_B (W \setminus Z) \to W \setminus Z$.
		\item If the closed immersion $f': X \to Y \times_B Z$ induced by $f$ is regular (e.g., if $f$ is regular), then the morphism $\pi^{-1}(Z) \to Z$ is the composite 
		\[
		\mathbb{V}_{f'} \stackrel{p}{\longrightarrow} X \stackrel{\psi}{\longrightarrow} Z 
		\]
		where $p$ is a torsor under the vector bundle associated to $\mathscr{N}_{f'}$.  
		\item If both $i$ and $f'$ are regular immersions (the latter happens if $f$ is a regular immersion), then 
		\[
		\pi: \op{D}(X,i,f) \longrightarrow W
		\]
		is a smooth morphism; if $Y$ is furthermore finitely presented, then so is $\pi$.
	\end{enumerate}
\end{prop}

\begin{proof}
	We set $Y' = \op{Bl}_{X}(Y \times_B W)$ and observe that $\pi$ factors as 
	\[
	\op{D}(X,i,f) \longrightarrow Y' \longrightarrow Y \times_B W \longrightarrow W
	\]
	where the first morphism is an open immersion, and $Y' \to Y$ is an isomorphism on the complement of $X$ \cite[\href{https://stacks.math.columbia.edu/tag/02OS}{Tag 02OS}]{stacks-project}.  The first statement then follows from \cite[Lemma 2.4]{MRR} which states that the exceptional divisor of the affine blow-up, i.e., the pre-image of the center of the affine blow-up, coincides with the pre-image of $Y \times_B Z$.  
	
	It follows from the final observation of the preceding paragraph that the restriction of $\pi$ to $Z$ also factors as the composite of the projection map from the exceptional divisor to the center of the blow-up followed by the map $\psi$ in \ref{construction:parameterizeddeformations}.  In that case, the second statement is \cite[Proposition 2.9(1)]{MRR}; the assertion that the exceptional divisor is a torsor under a vector bundle follows by inspecting the proof of that result.  
	
	For the third statement, observe that the smoothness of $\pi$ under the stated hypotheses is \cite[Proposition 2.16(4)]{MRR}.  If $Y$ is finitely presented, then $Y \times_B W \to W$ is as well.  In that case, $\pi$ is finitely presented by \cite[Proposition 2.16(1)]{MRR}.  
\end{proof}

\begin{rem}
	As observed in \cite[\S 2.3]{MRR}, all of the results in \cite[\S 2]{MRR} hold under a weaker hypothesis than stated above.  In particular, Proposition~\ref{prop:smoothnessofparameterizeddeformationspaces} holds replacing the regularity assumption on $i$ and $f'$ by $H_1$-regularity in the sense of \cite[\href{https://stacks.math.columbia.edu/tag/063D}{Tag 063D}]{stacks-project}.
\end{rem}

\begin{cor}[Deformation to the normal cone]
	\label{cor:deformationtothenormalcone}
	If $B$ is a base scheme and $i: X \to Y$ is a closed immersion of finitely presented smooth $B$-schemes, 
	the morphism $\pi: \op{D}(X,Y) \to \aone_B$  (see \textup{Example~\ref{ex:deformationspace}}) enjoys the following properties:
	\begin{enumerate}[noitemsep,topsep=1pt]
		\item The morphism $\pi^{-1}({\mathbb A}^1_B \setminus 0) \to {\mathbb A}^1_B \setminus 0$ is isomorphic to the projection map $Y \times_B {\mathbb A}^1_B \setminus 0 \to {\mathbb A}^1_B \setminus 0$.
		\item The $B$-scheme $\pi^{-1}(0)$ is isomorphic to $\mathbb{V}_i$, the total space of the normal bundle to $i$.
		\item The morphism $\pi$ is a smooth morphism; if $Y$ is finitely presented, then so is $\pi$.
	\end{enumerate}
\end{cor}

\begin{proof}
	We check the hypotheses of Proposition~\ref{prop:smoothnessofparameterizeddeformationspaces} are satisfied.  Indeed, here we take $i: Z \to W$ as the closed immersion $B \hookrightarrow \aone_B$, which is a Cartier divisor and a regular immersion.  The map $f': X \to Y$ is also a regular immersion by assumption, and the result follows.  The only thing that isn't immediate is that $\mathbb{V}_i$ is the normal bundle rather than a torsor under it.  However, the exact sequence of conormal sheaves associated with the sequence of closed immersions $X \hookrightarrow Y \hookrightarrow Y \times_B \aone$ is split by the projection $Y \times_B \aone \to Y$, so the final assertion follows from (the proof of) \cite[Proposition 2.9(2)]{MRR}.
\end{proof}

Unwinding the definitions of blow-up algebras, it is frequently possible to write explicitly defining equations for the parameterized deformation spaces described above.  In particular, one realizes that many varieties presented explicitly by equations can be realized in terms of the parameterized deformation space construction.

\begin{ex}
	\label{prop:equationsZprincipal}
	Assume $B = \Spec k$ is the spectrum of  an arbitrary base ring $k$ and pick coordinates $x_1,\ldots,x_n$ on ${\mathbb A}^n_B$. Assume that $i: Z \hookrightarrow {\mathbb A}^n_B$ is a closed immersion of $B$-schemes defined by an element $g$, that $Y = \Spec R$ is a smooth affine $k$-scheme and that $X' \hookrightarrow Y$ is a smooth closed subscheme defined by a regular ideal $I=(f_1,\ldots,f_c)$.  Then  $f: X = X' \times_B Z \hookrightarrow Y \times_B {\mathbb A}^n$ is a regular closed immersion whose image is the smooth closed subscheme of $Y \times_B {\mathbb A}^n_B$ defined by the regular sequence $(f_1,\ldots,f_c,g)$. The scheme $Y \times_B Z$ is a Cartier divisor in $Y \times_B {\mathbb A}^n_B$ cut out by the single element $g$.  Unwinding the definitions, we see that the affine scheme of Construction~\ref{construction:parameterizeddeformations} is given explicitly by: 
	\[
	\D(X,i,f) = \Spec R[x_1,\ldots,x_n][I/g]=\Spec R[x_1,\ldots,x_n,t_1,\ldots,t_c]/(f_1 - t_1g,\ldots,f_c - t_c g).
	\]
\end{ex}

\subsection{The purity isomorphism}
We will refer to a closed immersion of smooth $B$-schemes $X \subset Y$ as a smooth pair.  By a morphism of smooth pairs $f: (X \subset Y) \to (X' \subset Y')$, we will mean a morphism of $B$-schemes $f: Y \to Y'$ such that $f$ restricts to a morphism of $B$-schemes $X \to X'$.  A morphism of smooth pairs is {\em transversal} if $X = f^{-1}(X')$ and the induced map $\varphi: \mathscr{N}_{X/Y} \to f^* \mathscr{N}_{X'/Y'}$ is an isomorphism.  Having introduced the above notions, the homotopy purity theorem takes the following form (we state it here because we could not find the statement at this level of generality in the literature; the original statement is \cite[\S 3 Theorem 3.23]{MV}).

\begin{thm}[Homotopy purity]
\label{thm:purity}
Given a smooth pair $X \subset Y$ in $\Sm_B$, there is a canonical isomorphism in $\ho{B}$ of the form
\[
X/(X \setminus Y) \cong \Th(\mathscr{N}_{X/Y}).
\]
Given a transversal morphism of pairs $f: (X \subset Y) \to (X' \subset Y')$, there is a homotopy commutative square of the form
\[
\xymatrix{
X/(X \setminus Y) \ar[r]\ar[d] & \Th(\mathscr{N}_{X/Y}) \ar[d]\\
X'/(X' \setminus Y') \ar[r] & \Th(\mathscr{N}_{X'/Y'}),
}
\]
where the vertical maps are induced by $f$.
\end{thm}

\begin{proof}
This follows immediately from the proof of \cite[Theorem 3.23]{HoyoisEquivariant}.  There are two remarks to make: the additional hypotheses (implicitly) placed on $B$ in \cite{HoyoisEquivariant} stem from the presence of the action of a group scheme.  Moreover, the fact that deformation to the normal cone construction has the required properties over an arbitrary base scheme $B$ follows from Corollary~\ref{cor:deformationtothenormalcone}.  The transversality hypothesis is precisely what is needed to obtain a morphism of Thom spaces; the classical functoriality statement may be found at \cite[Lemma 2.1]{VoevodskyMC}.
\end{proof}

\section{$\mathbb{A}^1$-homotopy types of generic weak equivalences}

Fix a base scheme $B$, and suppose $W$ is a pointed smooth $B$-scheme.  Recall from the introduction that a smooth morphism $\pi: \mathfrak{X} \to W$ is called a generic $\aone$-weak equivalence if there exists an open dense subscheme $U$ of $W$ such that $\pi|_U: \mathfrak{X}|_U \to U$ is an $\aone$-weak equivalence.  For $U$ maximal with this property, we will refer to $W \setminus U$ as the {\em degeneration locus} of $\pi$.  Our aim is to describe the $\aone$-homotopy type of $\mathfrak{X}$ in terms of the degeneration locus $W \setminus U$ of $\pi$.

\subsection{Degenerations over a subscheme} 
\label{ss:degenerationoverasubscheme}
The goal of this section is to establish the following result.

\begin{thm}
\label{thm:mainmultiple}
Fix a finitely presented, pointed, smooth $B$-scheme $(W,w)$ that is $\aone$-contractible.  Assume $\mathfrak{X}$ is a $B$-scheme, and $\pi: \mathfrak{X} \to W$ is a smooth morphism of $B$-schemes admitting a section $s$ (point $\mathfrak{X}$ by $s(w)$).  Assume $Z$ is a finitely presented smooth $B$-scheme, $i: Z \to W$ is a closed immersion, and let $j: U := W \setminus Z \to W$ be the complementary open immersion.  Write $\mathfrak{X}|_U$ for the restriction of $\mathfrak{X}$ to $U$ and $\pi|_U$ for the restricted morphism, and let $i': \mathfrak{X}_Z \to \mathfrak{X}$ be the base-change of $i$ along $\pi$.  The morphisms $i$ and $i'$ are regular immersions. If the morphism $\pi|_U: \mathfrak{X}_U \to U$ is an $\aone$-weak equivalence, then there is a split cofiber sequence 
\[
\Th(\mathscr{N}_{i}) \longrightarrow \Th(\mathscr{N}_{i'}) \longrightarrow \mathfrak{X},
\]
where $\mathscr{N}_i$ and $\mathscr{N}_{i'}$ are the normal bundles of the corresponding regular immersions.
\end{thm}

\begin{proof}
For notational convenience, we will suppress $B$ from the notation. Consider the section $s: W \to \mathfrak{X}$.  By assumption, the restriction of $s$ to $\mathfrak{X}_U$ coincides with the inclusion $U \to \mathfrak{X}_U$.  Since $\pi|_U$ is an $\aone$-weak equivalence by construction and since $s$ is a section of $\pi$, it follows that $s|_U$ is an $\aone$-weak equivalence by the $2$ out of $3$ property for $\aone$-weak equivalences.  The section $s$ then induces a Cartesian square of the form
\begin{equation}
\label{equation:smoothpairdiagram}
\xymatrix{
U \ar[r]\ar[d] & W \ar[d] \\
\mathfrak{X}_U \ar[r]& \mathfrak{X}.
}
\end{equation}
By assumption, the complementary closed immersions form a transversal morphism of pairs.

Note that $\mathfrak{X}_Z$ is smooth over $B$ since it is smooth over $Z$.  Since $Z \to W$ is a closed immersion of smooth schemes, it follows that the base-change of $\pi$ along $i$, i.e., $\mathfrak{X}_Z \to \mathfrak{X}$, is a closed immersion of smooth $B$-schemes.  As a consequence, $i': \mathfrak{X}_Z \to \mathfrak{X}$ is a regular immersion and thus has a well-defined normal sheaf $\mathscr{N}_{i'}$.

By homotopy purity in Theorem \ref{thm:purity}, there is thus an $\aone$-weak equivalence $\mathfrak{X}/\mathfrak{X}_U \sim \Th(\mathscr{N}_{i'})$.  Since the map $U \to \mathfrak{X}_U$ is an $\aone$-weak equivalence, we also conclude that
\[
\hocofib(U \to \mathfrak{X}) \cong \Th(\mathscr{N}_{i'}).
\]
This description corresponds to computing the homotopy colimit of the morphism $U \to \mathfrak{X}$ going along the left and bottom edge of \eqref{equation:smoothpairdiagram}.

Now, we compute the homotopy colimit of $U \to \mathfrak{X}$ going around the top and right edge of \eqref{equation:smoothpairdiagram}.  To this end, we appeal to Proposition~\ref{prop:octahedralaxiom}, which tells us there is a cofiber sequence of the form
\[
\hocofib(U \to W) \longrightarrow \hocofib(U \to \mathfrak{X}) \longrightarrow \hocofib(W \to \mathfrak{X}).
\]
The homotopy purity isomorphism applied to the closed immersion $Z \to W$ shows that the first term in the above cofiber sequence is $\aone$-weakly equivalent to $\Th(\mathscr{N}_{i})$.  Because $W$ is $\aone$-contractible, Lemma~\ref{lem:leftproper} implies that the last term is $\aone$-weakly equivalent to $\mathfrak{X}$ pointed by any point in the image of $W$, e.g., $s(w_0)$.

Combining all the observations above, we conclude that there is a cofiber sequence of the form
\[
\Th(\mathscr{N}_{i}) \longrightarrow \Th(\mathscr{N}_{i'}) \longrightarrow \mathfrak{X}.
\]
Note that the first map is split by $\pi$ and is thus evidently a cofibration; in fact, $\mathfrak{X}$ is the actual cofiber of the first map as well.
\end{proof}

\begin{rem}
	\label{rem:gluing}
	Suppose $B$ is qcqs base scheme, $W$ is a finitely presented smooth $B$-scheme, and $\pi: \mathfrak{X} \to W$ is a smooth morphism.  Assume $i: Z \to W$ is a closed immersion with open complement $j: U \to W$.  In that case, there is a homotopy cocartesian gluing square (see \cite[Theorem 2.21]{MV} or \cite[Theorem 4.18]{HoyoisEquivariant} for a statement in this generality) of the form
	\[
	\xymatrix{
	j_{\sharp}j^* \mathfrak{X} \ar[r]\ar[d] & \mathfrak{X} \ar[d] \\
	U \ar[r] & i_*\mathfrak {X}_Z.
	}
	\]
	Unwinding the definitions, $j_{\sharp}j^* \mathfrak{X}$ is simply $\mathfrak{X}|_U$ considered as a $W$-scheme.  If $\pi|_U$ is an $\aone$-weak equivalence, then the map $j_{\sharp}j^* \mathfrak{X} \to U$ is an $\aone$-weak equivalence also.  The top horizontal map is a cofibration, and so left properness of the $\aone$-local model structure implies that $\mathfrak{X} \to i_* \mathfrak{X}_Z$ is an $\aone$-weak equivalence as well.  In essence, Theorem~\ref{thm:mainmultiple} shows that the additional control afforded by the section $s$ and the smoothness of $Z$ allows us to identify the homotopy type of $i_*\mathfrak {X}_Z$ more explicitly.
\end{rem}

\begin{rem}
\label{rem:trivializations}
If $s: X \to B$ is the structure morphism of a smooth $B$-scheme, and if $\mathscr{E}$ is a finite rank locally free sheaf of $\mathscr{O}_X$-modules, then by a {\em weak trivialization} of $\mathscr{E}$ we will mean a pair $(\mathscr{E}_0,\varphi)$ consisting of a finite rank locally free sheaf $\mathscr{E}_0$ of $\mathscr{O}_B$-modules and an isomorphism $\varphi: \mathscr{E} \isomt s^* \mathscr{E}_0$; we will say that $\mathscr{E}$ is {\em constant} if it admits a weak trivialization.  Extending \cite[\S 3 Proposition 2.7(2)]{MV}, if $\mathscr{E}$ is a constant locally free sheaf of finite rank with a chosen weak trivialization $(\mathscr{E}_0,\varphi)$, then
\[
\Th(\mathscr{E}) \cong \Th(\mathscr{E}_0) \sma_B X_+.
\]
Indeed, this observation follows from that one by simply choosing an open cover of $X$ along which $\mathscr{E}_0$ is trivial.
\end{rem}

While Theorem~\ref{thm:mainmultiple} does give some insight into the $\aone$-homotopy type of $\mathfrak{X}$, it will be more useful with additional hypotheses in place.  For example, we may use Theorem~\ref{thm:mainmultiple} to construct $\aone$-contractible smooth schemes via the following corollary.

\begin{cor}
\label{cor:multiplea1contractibles}
Assume $\pi: \mathfrak{X} \to W$ and $Z \to W$ are as in \textup{Theorem~\ref{thm:mainmultiple}}.  If the map $\mathfrak{X}_Z \to Z$ is an $\aone$-weak equivalence, then $\mathfrak{X}$ is $\aone$-contractible.
\end{cor}

\begin{proof}
In the terminology of \cite[Definition 3.17]{HoyoisEquivariant}, the smooth pairs $(W,Z)$ and $(\mathfrak{X},\mathfrak{X}_Z)$ are weakly excisive, so the assumption that $\mathfrak{X}_Z \to Z$ is an $\aone$-weak equivalence implies that the induced map $\Th(\mathscr{N}_{Z/W}) \to \Th(\mathscr{N}_{\mathfrak{X}_Z/\mathfrak{X}})$ is an $\aone$-weak equivalence by properness of the $\aone$-local model structure.
\end{proof}

\subsection{A special case: smooth fibrations over affine spaces}
\label{ss:degenerationover0}
We now reconsider Theorem~\ref{thm:mainmultiple} under more stringent hypotheses on $Z$ and $W$.  In particular, if $W = {\mathbb A}^n_B$ and $Z = 0 \subset {\mathbb A}^n_B$, we may obtain a  more precise description of the $\aone$-homotopy type of $\mathfrak{X}$.  The following result establishes Theorem~\ref{thmintro:onedegeneratefiber} from the introduction.

\begin{thm}
\label{thm:main}
Assume $\mathfrak{X}$ is a $B$-scheme, and $\pi: \mathfrak{X} \to {\mathbb A}^n_B$ is a smooth morphism admitting a section $s$ such that, if $U := \pi^{-1}({\mathbb A}^n \setminus 0)$, then $s|_{{{\mathbb A}^n_B \setminus 0}}: {{\mathbb A}^n_B \setminus 0} \to U$ is an $\aone$-weak equivalence.  If $\mathfrak{X}_0 := \pi^{-1}(0)$, then there is a canonical (pointed) $\aone$-weak equivalence
\[
{\pone}^{\sma n} \wedge (\mathfrak{X}_0,s(0)) \sim (\mathfrak{X},s(0)).
\]
\end{thm}

\begin{proof}
Since $\pi$ is smooth, it follows that $\mathfrak{X}_0$ is a smooth $B$-scheme.  Moreover, since $\pi$ is flat, by choosing coordinates $t_1,\ldots,t_n$, i.e., an identification ${\mathbb A}^n_B = \Spec \mathscr{O}_B[t_1,\ldots,t_n]$, it follows that $t_1,\ldots,t_n$ is a regular sequence on $\mathfrak{X}$ and $\mathfrak{X}_0$ is defined by the vanishing of $t_1,\ldots,t_n$. In particular, the normal sheaf $\mathscr{N}_{\mathfrak{X}_0/\mathfrak{X}}$ is equipped with a corresponding trivialization (we will call this the canonical trivialization in what follows).  For notational clarity, we henceforth suppress $B$ from the notation (so, e.g., the zero sphere $S^0$ is $B_+$, i.e., a disjoint union of two copies of $B$).

Appealing to Theorem~\ref{thm:mainmultiple} and specializing the notation as necessary, we obtain the cofiber sequence of the form
\[
\Th(\mathscr{N}_{0/{\mathbb A}^n}) \longrightarrow \Th(\mathscr{N}_{\mathfrak{X}_0/\mathfrak{X}}) \longrightarrow \mathfrak{X}.
\]
The canonical trivialization determines an $\aone$-weak equivalence
\[
\Th(\mathscr{N}_{\mathfrak{X}_0/\mathfrak{X}}) \sim {\pone}^{\sma n} \sma (\mathfrak{X}_0)_+
\]
(see \cite[\S 3 Proposition 2.17]{MV} or Remark~\ref{rem:trivializations}).  The restriction of the canonical trivialization of $\mathscr{N}_{\mathfrak{X}_0/\mathfrak{X}}$ to $\mathscr{N}_{0/{\mathbb A}^n}$ yields a corresponding trivialization of that normal sheaf and an identification 
\[
\Th(\mathscr{N}_{0/{\mathbb A}^n}) \cong {\pone}^{\sma n} \wedge S^0 = {\pone}^{\sma n}.
\]

Using the compatibility of these trivializations and the conclusions of the two preceding paragraphs, we see that:
\begin{enumerate}[noitemsep,topsep=1pt]
\item There is a cofiber sequence of the form
\[
{\pone}^{\sma n} \longrightarrow {\pone}^{\sma n} \sma (\mathfrak{X}_0)_+ \longrightarrow \mathfrak{X}.
\]
\item The left-hand morphism in the above cofiber sequence is the map ${\pone}^{\sma n} = {\pone}^{\sma n} \wedge S^0 \to {\pone}^{\sma n} \wedge (\mathfrak{X}_0)_+$ corresponding to suspending the map $S^0 \to (\mathfrak{X}_0)_+$ given by sending the non-basepoint of $S^0$ to $s(0)$, and consequently
\item The left-hand morphism in the above cofiber sequence is split.
\end{enumerate}
Therefore, appealing to Proposition~\ref{prop:basepointedsmash}, there is an $\aone$-weak equivalence
\[
{\pone}^{\sma n} \sma (\mathfrak{X}_0) \isomto \mathfrak{X},
\]
with base-points as stated.
\end{proof}

As a variant of the above result, we may also analyze smooth morphisms $\mathfrak{X} \to \aone_B$ that have multiple degenerate fibers.

\begin{thm}
\label{thm:multipleaonebase}
Assume $\mathfrak{X}$ is a $B$-scheme, and $\pi: \mathfrak{X} \to {\mathbb A}^1_B$ is a smooth morphism admitting a section $s$.  Assume there exists a closed immersion of $B$-schemes $i: Z \hookrightarrow \aone_B$ with open complement $U \subset \aone_B$
having the following properties:
\begin{enumerate}[noitemsep,topsep=1pt]
\item The immersion $i$ is defined by a monic polynomial $g \in \mathscr{O}_B[t]$ that factors as a product of linear factors with roots in $B$.
\item There exists a closed subscheme $\tilde{Z} = B \sqcup \cdots \sqcup B$ having the same open complement as $Z$ such that $\tilde{Z}$ is smooth over $B$.
\item The morphism $\pi: \pi^{-1}(U) \to U$ is an $\aone$-weak equivalence.
\end{enumerate}
Write $b_i$, $i = 1,\cdots,r$ for the inclusion of the $i$-th factor of $B$ in $\tilde{Z}$, $\mathfrak{X}_{b_i}$ for the base-change of $\pi$ along the morphism $B \to \aone_B$ determined by $b_i$, and point $\mathfrak{X}_{b_i}$ by $x_i$, obtained from $b_i$ and $s$.  In that case, there is a pointed $\aone$-weak equivalence
\[
\pone \wedge (\bigvee_{i=1}^r \mathfrak{X}_{b_i}) \sim \mathfrak{X},
\]
where $\mathfrak{X}$ is pointed by any base-point $x_i$.
\end{thm}

\begin{rem}
Before getting to the proof, let us observe that conditions in the theorem are easy to check when $B$ is an integral affine scheme.  Indeed, in that case, since $g$ is a product of linear factors, we can find a separable polynomial $\tilde{g}$ that divides $g$. The smoothness condition is equivalent to the assertion that the discriminant $\operatorname{disc}(\tilde{g})$ is a unit in each residue field at a closed point.
\end{rem}

\begin{proof}
Write $\mathfrak{X}_{\tilde{Z}}$ for the fiber product of $\tilde{Z}$ and $\mathfrak{X}$ over $\aone_Z$; note that the base-change of $\pi$ along $\tilde{Z} \to Z$ is a smooth morphism, so the fibers $\mathfrak{X}_{b_i}$ corresponding to the components of $\tilde{Z}$ are all smooth schemes.  Choose base-points as in the statement of the theorem.

In that case, we may appeal to Theorem~\ref{thm:mainmultiple}.  In particular, there is a cofiber sequence of the form
\[
\Th(\mathscr{N}_{\tilde{Z}/\aone_B}) \longrightarrow \Th(\mathscr{N}_{\mathfrak{X}_{\tilde{Z}}/\mathfrak{X}_Z}) \longrightarrow \mathfrak{X}.
\]
Note that $\tilde{Z}$ is cut out by a separable polynomial $\tilde{g}$ that divides $g$.  This element also defines a principal divisor on $\mathfrak{X}$ that cuts out $\mathfrak{X}_{\tilde{Z}}$.  It follows that the normal sheaf to each embedding is equipped with a trivialization, and these trivializations are compatible.  Therefore, the above cofiber sequence reads:
\[
\pone \sma \tilde{Z}_+ \longrightarrow \pone \sma (\mathfrak{X}_{\tilde{Z}})_+ \longrightarrow \mathfrak{X}.
\]
Now, if $(\mathscr{Y}_i,y_i)$ is a finite collection of pointed spaces, for any $j \in I$, we may write
\[
(\sqcup_{i \in I} \mathscr{Y}_i)_+ = (\mathscr{Y}_{j})_+ \vee (\sqcup_{i \in I \setminus \{ j \}} \mathscr{Y}_i)_+.
\]
Using this observation, the result follows by repeated appeal to Proposition~\ref{prop:basepointedsmash}.
\end{proof}

\begin{ex} \label{prop:globalthomsebastiani}
Assume that $B = \Spec k$ for some base ring $k$ and let $(X,\ast) \subset {\mathbb A}^{n+1}_k$ is a pointed smooth hypersurface defined by the vanishing of a polynomial $f \in k[z_1,\ldots,z_{n+1}]$ and a $k$-point $\ast$ of $X$.  Let $i:Z\hookrightarrow \mathbb{A}^1_k$ be a hypersurface defined by a polynomial $g = \prod_{i=1}^r (t-a_i)^{b_i}$, $a_i \in k$, $b_i\in \mathbb{Z}_{>0}$  such that the hypersurface $\tilde{Z}$ defined by the polynomial $\tilde{g}= \prod_{i=1}^r(t-a_i)$ is smooth over $k$. 

Then the hypersurface $\mathfrak{X}_g \subset {\mathbb A}^{n+3}_k$ defined by the equation $f - gx = 0$ is smooth over $k$ 
and there is an $\aone$-weak equivalence 
$$\mathfrak{X}_g \sim (\vee_{i=1}^r \pone) \sma (X,\ast).$$
  
Indeed, let $\pi:\mathfrak{X}_g \to \aone_k$ be the morphism induced by the projection onto the $x$-variable. The ring homomorphism $k[t,z_1,\ldots,z_{n+1},x])\to k[t,x]$ factors through a ring homomorphism $k[\mathfrak{X}_g] \to k[t,x]$ which yields a section of $\pi$.  The fibers of $\pi:\mathfrak{X}_g\to \mathbb{A}^1_k$ over the points in $\aone_k$ where $g$ vanishes are isomorphic to the total spaces of line bundles over $X$ (the normal bundles to the embedding of the given component).  In particular, each such fiber is $\aone$-weakly equivalent to $X$. The claimed $\aone$-weak equivalence thus follows from Theorem~\ref{thm:multipleaonebase}, together with the identification $\pone \sma (\vee_{i=1}^r (X,x)) \sim (\vee_{i=1}^r \pone) \sma (X,x)$. 
\end{ex}

\begin{rem}
For this remark, we work over the complex numbers.  If $f({\bf x})$ and $g({\bf y})$ are polynomials in $m$ and $n$ variables, then we write $f \oplus g$ for the polynomial $f({\bf x}) + g({\bf y})$ in $n+m$-variables.  The classical ``global" Thom--Sebastiani theorem \cite[Theorem 2.4]{Nemethi} states that the homotopy type of general fibers of $f \oplus g$ is the join of the homotopy types of the general fiber of $f$ and the general fiber of $g$.  A straightforward analog of this theorem in $\aone$-homotopy theory cannot be true because the $\aone$-homotopy type of the fibers of $f$ can vary rather wildly.  Example~\ref{prop:globalthomsebastiani} provides one version of such a theorem in $\aone$-homotopy theory.  The general fiber of $gt$ is $\Spec k[x,\frac{1}{g}]  = \aone_k \setminus \{x_1,\ldots,x_r\} \subset \aone_k$ and the join of $\aone_k \setminus \{x_1,\ldots,x_r\}$ and $(X,x)$ coincides with the wedge sum of a number of copies of $\pone$ with $(X,x)$ by appeal to the purity isomorphism.
\end{rem}

\section{Examples and applications}
In this section, we put the results of the preceding sections together and study applications.  For example, Section~\ref{ss:deformationspacesmodelsuspensions} shows that iterated $\pone$-suspensions of smooth schemes admit smooth models in a rather great generality.  The remainder of the section is concerned with building other schemes with controlled $\aone$-homotopy types.

\subsection{Deformation spaces model suspensions}
\label{ss:deformationspacesmodelsuspensions}
Suppose $(X,x)$ is a pointed smooth $k$-scheme.  If $Y$ is a smooth $k$-scheme and $i: X \hookrightarrow Y$ is a closed immersion, we will point $Y$ by $i(x)$.  If $\D(X,Y)$ is the deformation space of $i$ (see Example~\ref{ex:deformationspace}), then there is a canonical closed immersion $X \hookrightarrow {N}_{X/Y} \hookrightarrow \D(X,Y)$, where ${N}_{X/Y}$ is the total space of the normal sheaf $\mathscr{N}_{X/Y}$, and we point $\D(X,Y)$ by the image of $x$ under this composite morphism.  Note that the choice of base-point determines a section of the morphism $\D(X,Y) \to \aone$.  With this convention, we may now state the following result.

\begin{thm}
\label{thm:deformationspacesmodelsuspension}
Assume $B$ is a base scheme and $(X,x)$ is a pointed smooth $B$-scheme.  If there exists an $\aone$-contractible smooth $B$-scheme $Y$ and a (pointed) closed immersion $X \hookrightarrow Y$, then there is a pointed $\aone$-weak equivalence:
\[
\Sigma_{\pone} X \longrightarrow \D(X,Y).
\]
\end{thm}

\begin{proof}
By appeal to Corollary~\ref{cor:deformationtothenormalcone}, we see that the projection $\D(X,Y) \to \aone_B$ satisfies the hypotheses of Theorem~\ref{thm:main} (the section arises from the base-point).  In that case, we conclude that $\Sigma_{\pone} \D(X,Y)_0 \sim \D(X,Y)$, where $\D(X,Y)_0$ is the scheme-theoretic fiber of the projection over $0$.  However, $\D(X,Y)_0$ is the total space of the normal bundle to the embedding $X \hookrightarrow Y$, and therefore the projection map $\D(X,Y)_0 \to X$ is a (pointed) $\aone$-weak equivalence.  Thus, we conclude that $\Sigma_{\pone} \D(X,Y)_0 \to \Sigma_{\pone} X$ is a (pointed) $\aone$-weak equivalence as well.
\end{proof}

The following result gives a rather general condition when the hypotheses of the preceding theorem are satisfied.

\begin{prop}
\label{prop:embedintoaffines}
If $B$ is a qcqs base scheme, and $X$ is a finitely presented smooth affine $B$-scheme, then there exists a finitely presented smooth $\aone$-contractible $B$-scheme $Y$ and a closed immersion of $B$-schemes $X \to Y$.
\end{prop}

\begin{proof}
In light of the finite presentation hypotheses, there exists a Noetherian scheme $B_0$, a morphism $B \to B_0$ and a finitely-presented affine $B_0$-scheme $X_0$ such that $X = X_0 \times_{B_0} B$ \cite[\href{https://stacks.math.columbia.edu/tag/01Z6}{Tag 01Z6}]{stacks-project} and the structure morphism $X_0 \to B_0$ may also be assumed affine.  In that case, the result is classical: covering $B_0$ by affines, we may glue together embeddings corresponding to local generators to obtain a vector bundle over $B_0$ into which $X_0$ embeds (this is evidently a finitely presented $B$-scheme).  The base-change of this vector bundle to $B$ is again a vector bundle, and thus $X$ comes equipped with a closed immersion into this vector bundle.
\end{proof}

We can also extend Theorem~\ref{thm:deformationspacesmodelsuspension} to the situation where a space admits an embedding into an affine space up to homotopy.  To that end, write $\Sm^{\aff}_B$ for the full subcategory of $\Sm_B$ consisting of schemes that are affine in the absolute sense.  

\begin{defn}
\label{defn:jouanoloudevice}
A {\em Jouanolou device} for $X\in \Sm_B$ consists of a pair $(\tilde{X},\varphi)$ where $\tilde{X} \in \Sm_B^{\aff}$ and $\varphi: \tilde{X} \to X$ is a morphism making $\tilde{X}$ into a torsor under a vector bundle over $X$.
\end{defn}

Jouanolou observed \cite[Lemme 1.5]{Jouanolou} that quasi-projective schemes always possess affine vector bundle torsors; Thomason extended this fact, and we recall these results here.

\begin{prop}[Jouanolou--Thomason homotopy lemma]
\label{prop:jouanolouthomason}
Suppose $B$ is a qcqs  base scheme.
\begin{enumerate}[noitemsep,topsep=1pt]
\item If $X \in \Sm_B$ admits an ample family of line bundles, then $X$ admits a Jouanolou device.
\item If $B$ is furthermore quasi-compact and regular, and $X \in \Sm_B$ has affine diagonal and is quasi-compact over $B$, then $X$ admits a Jouanolou device.
\end{enumerate}
\end{prop}

\begin{proof}
The first assertion is a restatement of \cite[Proposition 4.4]{Weibel} in our context.  For the second assertion, since $X$ has an affine diagonal, the structure morphism is quasi-separated.  Since the structure morphism is quasi-compact by assumption, it follows that $X$ is finitely presented over $B$.

Since $X$ is finitely presented and smooth over $B$, which is regular, we claim that $X$ is regular.  To see this, first, observe that $X$ is automatically quasi-compact.  To check its local rings are regular, we may reduce to the affine case, so $B = \Spec R$ is affine and $X = \Spec S$, so $\varphi: R \to S$ is a smooth ring map with $A$ regular.  In particular, this means that $S$ is a Noetherian $R$-algebra.  Moreover, smoothness implies $\varphi$ is a flat ring map with regular fibers.  Let ${\mathfrak q} \subset S$ be a prime ideal, and let ${\mathfrak p}$ be its pre-image in $R$.  Choose a regular sequence $f_1,\ldots,f_n$ in $R_{{\mathfrak p}}$ that generates ${\mathfrak p}R_{{\mathfrak p}}$.  Since $R \to S$ is flat the image of $f_1,\ldots,f_n$ in $S$ is again a regular sequence.  Moreover, $S_{{\mathfrak p}}/(\varphi(f_1),\ldots,\varphi(f_n))$ is a fiber of $\varphi$ hence regular.  It follows that $S_{{\mathfrak p}}$ is regular, and thus so is $S_{{\mathfrak q}}$.

Finally, since $X$ is regular, it is automatically locally factorial by the Auslander--Buchsbaum theorem \cite[\href{https://stacks.math.columbia.edu/tag/0AG0}{Tag 0AG0}]{stacks-project} (in particular normal).  In that case, appeal to \cite[Proposition 1.3]{BrennerSchroer} implies that $X$ carries an ample family of line bundles.  Then, the second point follows from the first.
\end{proof}

\begin{ex}
\label{ex:affine-space-multiple-origin}
Some separation hypothesis is necessary for Proposition~\ref{prop:jouanolouthomason} to guarantee the existence of a Jouanolou device.  If $n \geq 1$, then write ${\mathbb A}^n_{2 \cdot 0}$ for the quasi-separated $k$-scheme given by the affine $n$-space with a doubled origin; while this scheme fails to be separated, it has affine diagonal if and only if $n = 1$.  

The scheme ${\mathbb A}^1_{2 \cdot 0}$ has a Jouanolou device; in fact, a ``standard" choice of a Jouanolou device is the hypersurface given by $xy = z(1+z)$ in ${\mathbb A}^3_k$.  In more detail, consider the inclusion of $\gm{} \subset \op{SL}_{2}$ as diagonal matrices with determinant $1$.  An explicit computation with invariants identifies the quotient $\op{SL}_2/\gm{}$ with the hypersurface in ${\mathbb A}^3$ defined by the equation $xy = z(1+z)$.  On the other hand, consider the map $\op{SL}_2 \to {\mathbb A}^2 \setminus 0$ corresponding to projection onto the first column; this map is $\gm{}$-equivariant for the action of $\gm{}$ on ${\mathbb A}^2 \setminus 0$ given by $t \cdot (x,y) = (tx,t^{-1}y)$, but is also a $\ga{}$-torsor.  The geometric quotient of $\gm{}$ acting on ${\mathbb A}^2 \setminus 0$ exists as a smooth scheme and is identified with ${\mathbb A}^1_{2 \cdot 0}$ by explicit computation: if we cover ${\mathbb A}^2 \setminus 0$ with the two open sets ${\mathbb A}^2 \setminus \{x = 0\}$ and ${\mathbb A}^2 \setminus \{y = 0\}$, then the function $xy$ is invariant and yields identifications ${\mathbb A}^2 \setminus \{x = 0\}/\gm{} \cong {\mathbb A}^1$ and ${\mathbb A}^2 \setminus 0 \setminus {y = 0}/\gm{} \cong {\mathbb A}^1$.  It follows that the $\ga{}$-torsor $\op{SL}_2 \to {\mathbb A}^2 \setminus 0$ descends to a $\ga{}$-torsor $\op{SL}_2/\gm{} \to {\mathbb A}^1_{2 \cdot 0}$ with the formulas given above.

By contrast, when $n > 1$, the scheme ${\mathbb A}^n_{2 \cdot 0}$ has strictly quasi-affine diagonal\footnote{The condition of having quasi-affine diagonal is equivalent to quasi-separateness because the diagonal morphism of a scheme is always an immersion \cite[\href{https://stacks.math.columbia.edu/tag/01KJ}{Tag 01KJ}]{stacks-project} and quasi-compact immersions are quasi-affine \cite[\href{https://stacks.math.columbia.edu/tag/02JR}{Tag 02JR}]{stacks-project}.} and does not have a Jouanolou device; this is, of course, related to Thomason's observation that such schemes do not possess the resolution property \cite[Exercise 8.6]{ThomasonTrobaugh}.  Indeed, take $k = \Z$ and suppose $\pi: X \to {\mathbb A}^n_{2 \cdot 0}$ is a torsor under a vector bundle.  Since a torsor under a vector bundle on an affine scheme admits a section, hence is a vector bundle, and vector bundles on ${\mathbb A}^n_{\Z}$ are trivial by the Quillen--Suslin theorem, the restriction of $\pi$ over either copy of ${\mathbb A}^n_{\Z}$ is isomorphic to a trivial bundle. On the other hand, since the inclusion ${\mathbb A}^n_{\Z} \setminus 0 \to {\mathbb A}^n_{\Z}$ has complement of codimension $\geq 2$, and ${\mathbb A}^n_{\Z} \setminus 0$ is normal, any isomorphism between the two restrictions $\pi$ over ${\mathbb A}^n_{\Z} \setminus 0$ extends over ${\mathbb A}^n_{\Z}$.  Thus, one obtains a global isomorphism of $X$ with a product ${\mathbb A}^n_{\Z} \times {\mathbb A}^n_{2 \cdot 0}$, in particular, $X$ is not affine.  
\end{ex}

With the Jouanolou--Thomason homotopy lemma in hand, we can extend Theorem~\ref{thm:deformationspacesmodelsuspension} to the situation of schemes that may be embedded in an $\aone$-contractible scheme up to $\aone$-homotopy.  

\begin{cor}
\label{cor:modelsofiteratedsuspension}
Assume $B$ is a regular affine base scheme.
\begin{enumerate}[noitemsep,topsep=1pt]
\item If $(X,x)$ is a quasi-compact smooth $B$-scheme with affine diagonal, then for any integer $i \geq 0$, the iterated $\pone$-suspension $\Sigma^i_{\pone}X$ has the $\aone$-homotopy type of a smooth $B$-scheme as well.
\item If $(X,x)$ is furthermore affine, then $\Sigma^i_{\pone}X$ admits a smooth affine model as well.
\end{enumerate}
\end{cor}

\begin{proof}
Under the assumptions on $B$, $X$ satisfies the hypotheses of Proposition~\ref{prop:jouanolouthomason}.  In particular, there exists a smooth $B$-scheme $\tilde{X}$ and vector bundle torsor $\tilde{X} \to X$ with affine total space $\tilde{X}$. The morphism $\tilde{X} \to X$ is an $\aone$-weak equivalence.

Since $(X,x)$ is pointed, we have a morphism $x: \Spec B \to X$.  If we base-change $\tilde{X}$ along $x$, we obtain an affine vector bundle torsor $\tilde{B} \to B$.  Since $B$ is affine by assumption, $\tilde{B} \to B$ is isomorphic to a vector bundle over $B$; fix such an isomorphism, and define a morphism $\tilde{x}: B \to \tilde{B}$ to be the image of the zero section under the isomorphism.  Composing this section with the canonical morphism $\tilde{B} \to \tilde{X}$ from the fiber product, we see that $(\tilde{X},\tilde{x}) \to (X,x)$ is a pointed $\aone$-weak equivalence.  Therefore, without loss of generality, we may replace $(X,x)$ by $(\tilde{X},\tilde{x})$ and assume that $X$ is affine.

In that case, fix a closed embedding $\iota: X \hookrightarrow Y$ for $Y$ some $\aone$-contractible smooth $B$-scheme; such an embedding exists by appeal to Proposition~\ref{prop:embedintoaffines}.  Granted this observation, the first point follows immediately by inductive appeal to Theorem~\ref{thm:deformationspacesmodelsuspension}.  Note also that if $Y$ is a smooth affine $B$-scheme, then the deformation space $\D(X,Y)$ is automatically a smooth affine $B$-scheme as well, so the second point follows also.
\end{proof}

This following observation generalizes \cite[Theorem 2]{AsokDoranFasel}. 
\begin{cor}
\label{ex:quadrics}
For every  integer $n \geq 0$, ${\pone}^{\sma n}$ and $\gm{} \sma {\pone}^{\sma n}$ have the $\aone$-homotopy type of smooth {\em separated} $\Z$-schemes. 
\end{cor}
\begin{proof} The ring homomorphism $\Z[t] \to \Z[t]/(t(1-t))$ defines a closed immersion of smooth schemes $S^0_{\Z} \to {\mathbb A}^1_{\Z}$.  Likewise, the ring homomorphism $\Z[t_1,t_2] \mapsto \Z[t_1,t_2]/(t_1t_2-1)$ defines a closed immersion of smooth schemes $\gm{} \to {\mathbb A}^2_{\Z}$. The assertion then follows from Corollary~\ref{cor:modelsofiteratedsuspension}. 
\end{proof}

\begin{rem}
Note that no separation hypothesis is imposed on objects of $\Sm_B$ in the construction of $\ho{B}$.  If we do not restrict our attention to smooth {\em separated} $B$-schemes,  affine $n$-space with doubled origin provides a model for ${\pone}^{\sma n}$ as a smooth $B$-scheme.  Indeed, the homotopy colimit of the diagram
\[
{\mathbb A}^n  \longleftarrow {\mathbb A}^n \setminus 0\longrightarrow {\mathbb A}^n
\]
coincides with the homotopy pushout of ${\mathbb A}^n \setminus 0 \to {\mathbb A}^n$ (contract the ${\mathbb A}^n$ on the left to a point), i.e., $\Th(\mathscr{N}_{0/{\mathbb A}^n})$, which becomes $\aone$-weakly equivalent to ${\pone}^{\sma n}$ after fixing a basis of the tangent space at $0$.
More generally, the affine space $n$-space with $m$-fold origin provides a model for $\bigvee^{m-1} {\pone}^{\sma n}$.  
\end{rem}

\begin{question}
Fix an infinite field $k$, and a smooth affine $k$-scheme $X$ of dimension $d$.  If $i > 0$ is an integer, what is the minimum dimension of a smooth affine model of $\Sigma^i_{\pone}X$?  
\end{question}

\begin{rem}
If $k$ is an infinite field, and $X$ is a smooth affine $k$-scheme of dimension $d$, then it is well-known that $X$ can be embedded as a closed subscheme of ${\mathbb A}^{2d+1}$ and that this bound is optimal \cite[Theorem 5.8]{BlochMurthySzpiro}. Given the choice of one such closed embedding $X\hookrightarrow {\mathbb A}^{2d+1}$, Theorem \ref{thm:deformationspacesmodelsuspension} provides a model $\D(X,{\mathbb A}^{2d+1})$ of $\Sigma_{\pone} X$ of dimension $2d+2$.   
\end{rem}

\subsection{Highly $\mathbb{A}^1$-connected hypersurfaces}
In this section, we analyze a variation of a construction due to Danielewski \cite{Danielewski, Fieseler} to produce hypersurfaces that are ``highly $\aone$-connected".  Indeed, up to change of variables, Danielewski studied the varieties $x^ny = z(1-z)$ as $n$ varies.  In the context of $\aone$-homotopy theory, these varieties are all Jouanolou devices for the affine line with doubled origin ${\mathbb A}^1_{2 \cdot 0}$ discussed previously in Example \ref{ex:affine-space-multiple-origin}, and Danielewski was interested in analyzing their isomorphism types (these varieties are all stably isomorphic).  The second author considered a significant generalization of the Danielewski construction in \cite{DuboulozDF,DuboulozDanielewskiVarieties} and the version we analyze here can be viewed as axiomatizing some of the key properties considered in those papers.

\subsubsection*{A family of generic $\aone$-weak equivalences}
\begin{construction}
\label{construction:degenerationoverafpaffine}
Suppose $k$ is a base ring.  Assume $Y \subset {\mathbb A}^n_k$ is defined by a finitely generated ideal $I = (f_1,\ldots,f_r) \subset k[x_1,\ldots,x_n]$.  Fix an integer $s$, and suppose $a_1,\ldots,a_s \in k[x_1,\ldots,x_n]$ are polynomials.  
For ${\bf a} = (a_1,\ldots,a_s)$ we define
\[
X_{I,{\bf a}} = \{ \sum_i t_i f_i = \prod_{j=1}^s(z - a_j) \} 
\subset 
{\mathbb A}^{n+r+1}_k .
\]
Let $\pi: X_{I,{\bf a}}\rightarrow \mathbb{A}^n_k$ be the morphism defined by the inclusion of $k[x_1,\ldots,x_n]$ into the coordinate ring of $X_{I,{\bf a}}$.
\end{construction}

\begin{prop}
\label{prop:degenerationoverasubschemeexample}
In the setting of \textup{Construction~\ref{construction:degenerationoverafpaffine}}, assume that $X_{I,{\bf a}}$ is flat over $k$ and that the functions $\{f_i\}_{i=1,\ldots,r},\{ a_j\}_{j=1,\ldots,s}$ are not zero-divisors.  Then the following statements hold about the morphism $\pi: X_{I,{\bf a}} \to {\mathbb A}^n_k$.
\begin{enumerate}[noitemsep,topsep=1pt]
\item The restriction of $\pi$ to ${\mathbb A}^n_k \setminus Y$ is a Zariski locally trivial smooth morphism with affine space fibers.
\item The morphism $\pi$ admits a section, and if we set $Y' = \Spec k[x_1,\ldots,x_n,z]/(\prod_{i=1}^s(z - a_i),I)$, the base-change of $\pi$ along the inclusion $Y \hookrightarrow {\mathbb A}^n_k$ factors as
    \[
    {\mathbb A}^n_{Y'} \longrightarrow Y' \longrightarrow Y.
    \]
\item The morphism $\pi$ is finitely presented and faithfully flat.
\item The morphism $\pi$ is smooth if and only if $Y' \to Y$ is \'etale.
\item If the morphism $\pi$ is smooth, then it factors through a smooth morphism $\pi': X_{I,{\bf a}} \to {\mathbb A}^n_{sY}$ where the target is the non-separated scheme obtained by gluing $s$ copies of ${\mathbb A}^n_k$ with the identity map along ${\mathbb A}^n_k \setminus Y$.
\end{enumerate}
\end{prop}

\begin{proof} Put $S=k[x_1,\ldots,x_n,t_1,\ldots,t_r,z]$, $f=\sum_i t_if_i-\prod_j(z-a_j)\in S$ and  $R=S/(f)$. 

 1. By assumption, $f_i$ is not a zero-divisor. Inverting any $f_i$ yields an isomorphism $$S[f_i^{-1}]\cong k[x_1,\ldots, x_n,f_i^{-1}][t_1,\ldots,\hat{t}_i,\ldots,t_r,z],$$ showing that the restriction of $\pi$ to the principal open set $D(f_i)$ of $X_{I,{\bf a}}$ is thus isomorphic to the trivial $\mathbb{A}^{r}$-bundle over the principal open set  $D(f_i)$ of $\mathbb{A}^n$. This provides an explicit local trivialization of $\pi$ over ${\mathbb A}^n_k \setminus Y$. 

 2.  Fix an integer $j \in \{ 1,\ldots,s\}$ and observe that sending $t_i, i = 1,\ldots,r$ to $0$ and setting $z = a_j$ defines a section of $\pi$.  
The factor ring $S'= k[x_1,\ldots,x_n,t_1,\ldots,t_r, z]/(\prod_{i=1}^s z-a_i,I)$  of $S$ defines the closed subscheme scheme $Y'\subset X_{I,{\bf a}}$. By definition, we have a commutative square of the form
\[
\xymatrix{
Y' \ar[r]\ar[d] & X_{I,{\bf a}} \ar[d]^{\pi} \\
Y \ar[r] & {\mathbb A}^n_k,
}
\]
where the horizontal morphisms are closed immersions.The restriction of the section $\pi$ to $Y$ yields a section $Y \to Y'$. The base-change of $\pi$ along $Y$ factors as
\begin{equation}
\label{eqn:factorization}
{\mathbb A}^n_{Y'} \longrightarrow Y' \longrightarrow Y,
\end{equation}
where ${\mathbb A}^n_{Y'}$ is the spectrum of the polynomial ring in the variables $t_1,\ldots,t_r$ over $S'$. 

3. That $\pi$ is surjective is immediate because it admits a section.  That $\pi$ has finite presentation is immediate from the definitions.  Thus, it remains to check that $\pi$ is flat.  Since $z$ and the $t_i$ are not zero-divisors and the $f_i$ and $a_i$ are not zero-divisors by assumption,  it follows that $f = \sum_i f_it_i  - \prod_{i=1}^s (z - a_i)$  is not a zero-divisor in $R$. By assumption $X_{I,{\bf a}}$ is flat over $k$, which means that $S$ is a flat $k$-module. Let $A$ be the localization of $k[x_1,\ldots,x_n]$ at a prime ideal ${\mathfrak p}$ and let $B$ be the localization of $R$ at a prime ideal $\mathfrak{q}$ lying above ${\mathfrak p}$. Note that by construction, $B$ is essentially finitely presented over $A$, and $A \to B$ is a flat ring map. Let $m$ be the maximal ideal of $A$. Since $f$ is not a zero-divisor in $R$, its image in $B/{\mathfrak m}B$ is not a zero-divisor. By appeal to \cite[\href{https://stacks.math.columbia.edu/tag/046Z}{Tag 046Z}]{stacks-project}, we conclude that $B/(f)$ is flat over $A$.  Since $A$ and $B$ were arbitrary, it follows that $\pi$ is flat.

4. Since $\pi$ is flat and of finite presentation by assumption, to check it is smooth, it suffices to check that its fibers are smooth.  This statement is immediate over points in ${\mathbb A}^n_k$ that are contained in the complement of $Y$, so it suffices to check smoothness over points of $Y$.  To this end, consider the factorization of $\pi$ from \eqref{eqn:factorization}.  If $Y'$ is \'etale over $Y$, then this composite is evidently smooth, which shows that $\pi$ is smooth. Conversely, if $\pi$ is smooth, then the composite morphism in \eqref{eqn:factorization} is smooth as well.  The first morphism in that factorization is always smooth and surjective.  The morphism $Y' \to Y$ is evidently of finite presentation.  Therefore, \cite[\href{https://stacks.math.columbia.edu/tag/02K5}{Tag 02K5}]{stacks-project} shows that $Y' \to Y$ must also be smooth.  In that case, considerations of relative dimension imply that $Y' \to Y$ must also have relative dimension $0$, in which case it is automatically \'etale. 

5. Under the assumption that $\pi$ is smooth, $Y' \to Y$ is \'etale by the preceding point.  We build the required factorization by gluing.  The statement is tautological if $s = 1$, so assume $s \geq 2$.  In that case, define $E_j$ to be the closed subscheme of $X_{I,{\bf a}}$ defined by $f_1 = \cdots = f_r = \prod_{i \neq j} (z-a_i) = 0$.  Set $X_j = X_{I,{\bf a}} \setminus E_j$.  Note that $\{ X_j \}_{j = 1,\ldots,s}$ forms an open cover of $X_{I,{\bf a}}$.  The restriction of $\pi$ to $X_j$ defines a smooth morphism $\pi_j: X_j \to {\mathbb A}^n_k$.  For any $j \neq j'$, the intersection $X_j \cap X_{j'}$ is $\pi^{-1}({\mathbb A}^n_k \setminus Y)$, so the restrictions of $\pi_j$ and $\pi_{j'}$ to $X_{j} \cap X_{j'}$ coincide with the restriction of $\pi$ to ${\mathbb A}^n_k \setminus Y$. It follows that the morphisms $\pi_j$ glue to yield the morphism $ \pi': X_{I,{\bf a}} \longrightarrow {\mathbb A}^n_{sY},$
which is smooth by construction.
\end{proof}

In the case where $I$ is a principal ideal, the $\aone$-homotopy type of the varieties from Proposition~\ref{prop:degenerationoverasubschemeexample} is relatively straightforward to identify.

\begin{lem}
\label{lem:Yprincipalaffinecover}
Assume $k$ is a normal Noetherian domain such that $Pic(k) = 0$ (e.g., $k$ is a UFD).  Consider the scheme $X_{I,{\bf a}}$ from \textup{Proposition~\ref{prop:degenerationoverasubschemeexample}}, where $I = (f)$ is principal and assume that $s \geq 2$.  Let $X_j := X_{I,{\bf a}} \setminus E_j$, where $E_j$ is the closed subscheme defined by the ideal $(I,\prod_{i \neq j}(z-a_i))$. Then the following hold:
\begin{enumerate}[noitemsep,topsep=1pt]
\item The scheme $X_j$ is isomorphic to ${\mathbb A}^{n+1}_k$.
\item If the morphism $\pi: X_{I,{\bf a}} \to {\mathbb A}^n_k$ is smooth, then the projection morphism $\pi': X_{I,{\bf a}} \to {\mathbb A}^n_{sY}$ is a Jouanolou device (in particular an $\aone$-weak equivalence).
\item If ${\mathbb A}^n \setminus Y$ has a $k$-point, then there is a (pointed) $\aone$-weak equivalence $X_{I,{\bf a}} \sim \bigvee_{i=1}^{s-1}(\Sigma {\mathbb A}^n \setminus Y)$.  If $Y_{red}$ is furthermore smooth over $\Spec k$, then $X_{I,{\bf a}} \sim (\vee_{i=1}^{s-1} \pone) \sma (Y_{red})_+$.
\end{enumerate}
\end{lem}

\begin{proof}
We first note that under our assumptions, every Zariski locally trivial $\mathbb{A}^1$-bundler over ${\mathbb A}^n_k$ is globally trivial. Indeed, recall that the automorphism group $Aut(\aone_k)$ is the affine group $\gm{}_k\ltimes\ga{}_k$. Since ${\mathbb A}^n_k$ is affine, $H^1_{\Zar}({\mathbb A}^n_k,\ga{}) = 0$ and since $k$ is Noetherian and normal, homotopy invariance of Picard groups \cite[Corollary 5.10]{BassMurthy} guarantees that $H^1_{\Zar}({\mathbb A}^n_k,\gm{})=Pic({\mathbb A}^n_k) = Pic(\Spec k)$ which is trivial by assumption. Therefore, the long exact sequence in non-abelian cohomology \cite[Proposition III.3.3.1]{Giraud} attached to the short exact sequence of sheaves of groups
\[
1 \longrightarrow \mathbb{G}_{a,k} \longrightarrow Aut(\aone_k) \longrightarrow \gm{}_k \longrightarrow 1
\]
allows us to conclude that $H^1_{\Zar}({\mathbb A}^n_k,Aut(\aone_k))$ is also trivial.

1. By the observation above, it suffices to show that $\pi_j$ is a Zariski locally trivial $\aone_k$-bundle.  By construction, the morphism $\pi_j: X_j \to {\mathbb A}^n_k$ is a Zariski locally trivial $\aone$-bundle over ${\mathbb A}^n \setminus Y$.  Likewise, the morphism $\pi_j^{-1}(Y) \to Y$ is a {\em trivial} $\aone$-bundle, by the factorization from Proposition~\ref{prop:degenerationoverasubschemeexample}. All fibers of $\pi$ are thus isomorphic to $\aone_k$.  Since ${\mathbb A}^n_k$ is a normal Noetherian scheme, \cite[Theorem]{KambayashiWright} implies that $\pi$ is a Zariski locally trivial $\aone_k$-bundle.

2. The fact that $\pi'$ is a torsor under a line bundle follows immediately from gluing and the fact that the morphism $\pi_j: X_j \to {\mathbb A}^n_k$ is a trivial  $\aone_k$-bundle by the preceding point.  

3. This follows from the previous point and a computation of the homotopy colimit of the diagram consisting of $s$ maps from ${\mathbb A}^n \setminus Y$ to $\ast$.  Indeed, the homotopy colimit of $s$-maps from any (pointed) scheme $W$ to $\ast$ is a wedge sum of $s-1$ copies of $\Sigma W$.  Then, the cofiber sequence ${\mathbb A}^n_k \setminus Y \to {\mathbb A}^n_k \to {\mathbb A}^n_k/({\mathbb A}^n_k \setminus Y)$ and $\aone$-contractibility of ${\mathbb A}^n_k$ shows that ${\mathbb A}^n_k/({\mathbb A}^n_k \setminus Y) \sim \Sigma {\mathbb A}^n_k \setminus Y$.  If $Y_{red}$ is smooth, then $Y_{red} \to {\mathbb A}^n_k$ is a regular immersion, and the normal bundle to this embedding has an explicit trivialization given by $f$.  In that case, homotopy purity implies that $\Sigma {\mathbb A}^n_k \setminus Y \cong \pone \sma (Y_{red})_+$.
\end{proof}

\begin{rem}
	\label{rem:isomorphismtype}
Assuming $Y_{red}$ is smooth, Lemma~\ref{lem:Yprincipalaffinecover} gives complete control of the $\aone$-homotopy type of $X_{I,{\bf a}}$.  Thus, for example, if $f$ and $f'$ define smooth plane curves in ${\mathbb A}^2$ whose Picard groups are non-isomorphic and $a_0 = 0, a_1 = 1$, then the affine surfaces
\[
X_{(f),{\bf a}}=\{ft-z(z-1)=0\} \quad \textrm{and}\;\;\;\;\;\;\; X_{(f'),{\bf a}}=\{f't-z(z-1)=0\}
\] 
are not isomorphic.  Note, however, that the isomorphism type of $X_{I,{\bf a}}$ does depend in general on the {\em scheme structure} of $Y$ as well.  For example, take $a_0 = 0, a_1 = 1$, and let $f \subset \aone_k$ be given by $x^r$.  In that case, $X_{I,{\bf a}}$ is the Danielewski hypersurface $x^r t = z(z-1)$ whose isomorphism type is known to be determined by $r$  \cite{Crachiola,Danielewski,DuboulozDF,Fieseler}.
\end{rem}

\begin{ex}\label{ex:quadricsagain}
	Suppose $B = \Spec k$ where $k$ is an integral domain.  Assume $P(z) = \prod_{i=1}^s (z-a_i)$ with the $a_i$ pairwise distinct elements of $k$ and furthermore assume that $\operatorname{disc}(P)$ is a unit in every residue field at a maximal ideal of $k$.  Given a vector ${\bf m} = (m_1,\ldots,m_n)$, $m_i \in \Z_{>0}$, consider the scheme
	\[
	Q_{{\bf m},P} := \{ \sum_{i=1}^n x_i^{m_i} t_i = P(z) \}.
	\]
	In the setting of Construction \ref{construction:degenerationoverafpaffine}, we have $Q_{{\bf m},P}=X_{I,{\bf a}}$ for $I=(x_1^{m_1},\ldots, x_n^{m_n})$ and $\mathbf{a}=(a_1,\ldots, a_s)$. The condition on $P(z)$ guarantees that  $\Spec k[z]/P$ is smooth over $\Spec k$, hence that $Q_{{\bf m},P}$ is a smooth $k$-scheme. Since the corresponding scheme $Y_{red}$ is simply the origin in $\mathbb{A}^n_k$, we obtain by Remark \ref{rem:highercodim-homotopy-type} that 
	\[
	Q_{{\bf m},P} \sim  (\vee_{i=1}^{s-1} {\pone}^{\sma c}) \sma (Y_{red})_+ \sim ({\pone}^{\sma n})^{\vee deg P - 1}.
	\]
	Observe that these varieties are $\aone$-$(n-1)$-connected smooth affine $2n$-folds.  The problem of isomorphism classification is thus reminiscent of (a non-compact version of) that studied in \cite{Wall}.
	
	By an evident change of variables, we may assume that $P(z)$ is of the form $zQ(z)$.  When $Q(z) = (1-z)$ and ${\mathbf m} = (1,\ldots,1)$, the variety $Q_{{\bf m},P}$ is the quadric $Q_{2n}$ whose $\aone$-homotopy type was studied in \cite{AsokDoranFasel} and the above construction gives another view of the proof of \cite[Theorem 2]{AsokDoranFasel}.  When $Q(z)$ is general and ${\mathbf m} = (1,1)$, the variety $Q_{{\bf m},P}$ was studied implicitly in \cite[Proof of Corollary 3.1]{ADBundle} as a generalization of a construction of Winkelmann. 
\end{ex}	
	
The following question can be viewed as a concrete generalization of the discussion of Remark~\ref{rem:isomorphismtype} in the context of Example~\ref{ex:quadricsagain}.

\begin{question}
In the situation of \textup{Example~\ref{ex:quadricsagain}}, if $P$ is fixed, then can one distinguish non-isomorphic varieties of the form $Q_{{\bf m},P}$ as ${\bf m}$ varies?
\end{question}

\begin{rem} For specific integral domains $k$, the condition on $P(z)$ in Example \ref{ex:quadricsagain} can be rather stringent. For example, if $k = \Z$, {\em no} polynomial of degree $> 2$ satisfies the hypotheses.  Indeed, given $P$ as in the statement of degree $\geq 3$, there exist $2$ roots in $\Z$ that must differ by $\geq 2$.  The standard expression of the discriminant as a product of differences of roots then shows that $\operatorname{disc}{P}$ takes values in $\Z \setminus \{ -1,0,1\}$.  As such, there exists a prime $p$ such that $\Spec k[z]/P$ has bad reduction modulo $p$.
\end{rem}

\subsection{Constructing $\mathbb{A}^1$-contractible smooth schemes}
The discussion of the preceding section also gives a way to produce many new examples of $\aone$-contractible strictly quasi-affine schemes generalizing the analysis from \cite[Theorem 3.1.1]{AsokDoranFasel}. These provide new instances of exotic $\aone$-contractible schemes (i.e.,  not isomorphic to affine k-spaces). 

\begin{prop}
\label{prop:highercodimensionexamples}
Consider the morphism $\pi: X_{I,{\bf a}} \to {\mathbb A}^n_k$ from \textup{Proposition~\ref{prop:degenerationoverasubschemeexample}} where $I \subset k[x_1,\ldots,x_n]$ has height $\geq 2$, the morphism $\pi$ is smooth and $Y_{red}$ is smooth over $k$.  Define $E_j$ to be the closed subscheme corresponding to the ideal $(I,\prod_{i \neq j}(z-a_i))$ and set $X_{j} := X_{I,{\bf a}} \setminus E_j$.
\begin{enumerate}[noitemsep,topsep=1pt]
\item The scheme $X_j$ is a strictly quasi-affine $\aone$-contractible smooth $k$-scheme.
\item The morphism $\pi': X_{I,{\bf a}} \to {\mathbb A}^n_{sY}$ is an $\aone$-weak equivalence.
\end{enumerate}
\end{prop}

\begin{proof}
Proposition~\ref{prop:degenerationoverasubschemeexample} implies that
the morphism $\pi_j$ is smooth and surjective.  Moreover, by the construction of $E_j$, the base-change of $\pi_j$ along the closed immersion $Y \to {\mathbb A}^n_k$ is simply the projection morphism ${\mathbb A}^n_{Y} \to Y$ because $E_j$ consists precisely of all but one of the components of ${\mathbb A}^n_{Y'}$.  In particular, the morphism $\pi_j^{-1}(Y) \to Y$ is always an $\aone$-weak equivalence.  Corollary~\ref{cor:multiplea1contractibles} thus guarantees that $X_j$ is $\aone$-contractible if $Y_{red}$ is smooth. Note that $X_j$ is quasi-affine but not affine under the hypothesis that $I$ has height $\geq 2$. 

For the second statement, observe that each of the morphisms $\pi_j$ is an $\aone$-weak equivalence.  Since $\pi'$ is the pushout of all the morphisms $\pi_j$ along their intersections, it follows that $\pi'$ is also an $\aone$-weak equivalence, which establishes the second point.  
\end{proof}

\begin{rem} \label{rem:highercodim-homotopy-type}
Since  $Y_{red}$ is smooth, we may apply homotopy purity to describe the $\aone$-homotopy type of ${\mathbb A}^n_{sY}$ along the lines of the final statement in Lemma~\ref{lem:Yprincipalaffinecover}.  The corresponding statement is slightly more complicated because the normal bundle to $Y_{red}$ in ${\mathbb A}^n_k$ need not be trivial.  Of course, if $Y_{red} \to {\mathbb A}^n_k$ is a codimension $c$ complete intersection, then the normal bundle comes equipped with a trivialization and then $$X_{I,{\bf a}} \sim {\mathbb A}^n_{sY} \sim (\vee_{i=1}^{s-1} {\pone}^{\sma c}) \sma (Y_{red})_+.$$
\end{rem}

\begin{rem}
In contrast with the situation for principal ideals $I$ considered in Lemma \ref{lem:Yprincipalaffinecover} (1), Proposition~\ref{prop:highercodimensionexamples}(1) implies that when $I$ has height $\geq 2$, the morphism $\pi': X_{I,{\bf a}} \to {\mathbb A}^n_{sY}$ is not an affine morphism.  This observation has consequences for a higher-dimensional variant of the Danielewski fiber product construction.  Indeed, suppose $I' \subset k[x_1,\ldots,x_n]$ is another ideal corresponding to a closed subscheme $Y' \subset {\mathbb A}^n_k$, and assume $Y'_{red} = Y_{red}$.  Choose ${\bf a}'$ such that $X_{I',{\bf a}'} \to {\mathbb A}^n_k$ is smooth.  In that case, we may form the fiber product $X_{I,{\bf a}} \times_{{\mathbb A}^n_{sY}} X_{I',{\bf a}'}$.  One may show that either projection is an $\aone$-weak equivalence, again by appealing to Corollary~\ref{cor:multiplea1contractibles}.  Nevertheless, neither projection is affine since $\pi'$ is smooth and surjective, and the affineness of a morphism is local in the fppf topology on the base. As a consequence, neither projection can be Zariski locally trivial.  
\end{rem}

\begin{ex}
With the same notation as in Proposition~\ref{prop:highercodimensionexamples}, assuming that $k$ is a normal domain, then the isomorphism type of $X_j$ is closely tied to the stable isomorphism type of $Y$.  Indeed, take two $\aone$-contractibles as above, say $X$ and $X'$, defined by subschemes $Y$ and $Y'$ (together with the choices of corresponding $a_i$ and $a_i'$).  In that case, the normality assumption guarantees that an isomorphism between $X$ and $X'$ extends to the hypersurfaces $\bar{X}$ and $\bar{X}'$ in which $X$ and $X'$ are open subschemes, and then restricts to an isomorphism $Y \times {\mathbb A}^r \cong Y' \times {\mathbb A}^r$.  Thus, if $Y$ and $Y'$ are not stably isomorphic, the varieties $X$ and $X'$ are not isomorphic.  Since non-isomorphic smooth curves of genus $g > 0$ are never stably isomorphic, by choosing non-isomorphic smooth curves of genus $g > 0$ in ${\mathbb A}^3$, we may produce many non-isomorphic $\aone$-contractible smooth schemes.  This produces many (e.g., positive dimensional moduli spaces) non-isomorphic strictly quasi-affine $\aone$-contractible smooth schemes of dimension $d \geq 4$.  
\end{ex}

\begin{question}
Assume $k$ is a base ring.  If $Y_{red}$ is not necessarily smooth, then is $X_j$ still an $\aone$-weak equivalence?  More generally, if $\pi: \mathfrak{X} \to {\mathbb A}^n_k$ is a smooth morphism of $k$-schemes whose fibers over closed points are affine spaces, is $\pi$ an $\aone$-weak equivalence?
\end{question}

\subsection{Topological contractibility revisited}
In \cite[Conjecture 5.3.11]{AsokOstvaer}, the first and third authors conjectured that if $X$ is a topologically contractible smooth complex affine variety, then $\Sigma^{n}_{\pone} X$ is $\aone$-contractible for some $n$ sufficiently large and that $n = 2$ was probably sufficient.  In view of Theorem~\ref{thm:deformationspacesmodelsuspension}, this conjecture can be made significantly more precise; the following proposition provides a non-conjectural result inspired by this circle of ideas.

\begin{thm}
\label{thm:aonecontractiblehypersurfaces}
Assume $k$ is a field having characteristic $0$, and $(X,x)$ is a pointed $H\Z$-acyclic smooth $k$-affine variety.
\begin{enumerate}[noitemsep,topsep=1pt]
	\item For any integer $N \geq 2$, any smooth model of ${\pone}^{\sma N} \sma X$ is an $\aone$-contractible smooth scheme; if $X$ is furthermore $\aone$-connected, then $N = 1$ suffices.
	\item If $X$ is a topologically $\Z$-acyclic smooth complex surface, then for any integer $N \geq 2$, ${\pone}^{\sma N} \sma X$ has the $\aone$-homotopy type of an $\aone$-contractible smooth scheme.
	\item If $X$ is given as the vanishing locus of a hypersurface defined by $f \in \cplx[x_1,\ldots,x_n]$, then for any integer $N \geq 2$ (or $1$ if $X$ is $\aone$-connected) and any $N$-tuple of integers $(a_1,\ldots,a_n)$ the hypersurfaces
	\[
	\sum_{i=1}^N u_i^{a_i}v_i = f
	\]
	are $\aone$-contractible.  
\end{enumerate} 
\end{thm}

\begin{proof}
For the first point, begin by observing that by appeal to  Theorem~\ref{thm:deformationspacesmodelsuspension} any suspension ${\pone}^{\sma s} \sma X$ admits a model as a smooth scheme.  Next, pointed, smooth $k$-schemes are $H\Z$-local by appeal to \cite[Lemma 4.1]{HKO} (this argument appeals to resolution of singularities, which is where the assumption on the characteristic of $k$ appears).  Then, arguing as in \cite[Theorem 4.2]{HKO}, one observes that $H\Z$-acyclicity guarantees that $X$ is $\pone$-stably $\aone$-contractible.  

Next, if either $N \geq 2$, or $N = 1$ and $X$ is $\aone$-connected, then ${\pone}^{\sma N} \sma X$ is at least $\aone$-$1$-connected.  By definition, ${\pone}^{\sma N} \sma X$ is at least $1$-effective for $N \geq 1$ in the sense that it lies in the subcategory of $\Spc_k$ generated under homotopy colimits by spaces of the form $\gm{} \sma U_+, U \in \Sm_k$.  In that case \cite[Theorem 1.3]{BachmannConservativity} guarantees that since ${\pone}^{\sma N} \sma X$ is $\pone$-stably contractible, it is already $\aone$-contractible.

For the second point, observe that topologically contractible smooth complex surfaces are $H\Z$-acyclic by \cite[Theorem 1]{AAcyclic}.  That ${\pone}^{\sma N} \sma X$ is $\aone$-contractible then follows from the first point and admits a smooth affine model by appeal to Corollary~\ref{cor:modelsofiteratedsuspension}(2).

For the third point, we proceed as follows.  Observe that if $f$ defines a smooth hypersurface in ${\mathbb A}^{r}$, then projecting onto $\aone$ with coordinate $u$ and arguing as in Proposition~\ref{prop:degenerationoverasubschemeexample} one concludes that the hypersurface $u^av = f$ is a model of the $\pone$-suspension of the hypersurface defined by $f$. The result then follows by a straightforward induction.
\end{proof}

\begin{rem}
	\label{rem:topologicalcontractibility}
	If $X$ is a topologically $\Z$-acyclic smooth complex surface, then $X$ is necessarily affine and if $X$ is not isomorphic to ${\mathbb A}^2_{\cplx}$ then $X$ necessarily has negative logarithmic Kodaira dimension \cite[Theorem 2.6]{Zaidenberg}.  On the other hand, $\aone$-connected smooth surfaces are log-uniruled by \cite[Theorem 4.7]{ChoudhuryRoy}, so they necessarily have negative Kodaira dimension \cite[Theorem 1.1]{KeelMcKernan}.  It follows that topologically $\Z$-acyclic smooth complex surfaces not isomorphic to ${\mathbb A}^2_{\cplx}$ are not $\aone$-connected.
\end{rem}

\begin{ex}
	\label{ex:generalizedtomdieckpetriesurfaces}
Assume $X$ is a topologically $\Z$-acyclic smooth complex surface.  If $X$ has logarithmic Kodaira dimension $1$, it was observed by Kaliman and Makar-Limanov \cite[\S 7 Theorem on p. 606]{KML} that $X$ can always be realized as a hypersurface.  Indeed, suppose $k,l,m$ are integers with $k,l \geq 2$, $m \geq 1$, $gcd(k,l) = 1$, and assume that $f,g \in \cplx[x]$ are polynomials such that $\deg f, \deg g < m$, $f(0) = g(0) = 1$, $f$ is arbitrary subject to the preceding condition and $g$ is uniquely determined by the condition that  
\[
p_{k,l,m,f}(x,y,z) = \frac{(z^mx + f(z))^k - (z^my+g(z))^l - z}{z^m}
\]
is a polynomial.  In that case, the variety $p_{k,l,m,f} = 0$ defines a topologically contractible hypersurface of logarithmic Kodaira dimension $1$, and every such variety is isomorphic to a hypersurface of this form.  In particular, Theorem~\ref{thm:aonecontractiblehypersurfaces} applies here, and we conclude that the hypersurface $X_{p,N}$ defined by
\[
\sum_{i=1}^N u_iv_i = p_{k,l,m,f}
\]
is $\aone$-contractible for $N \geq 2$; this observation can be viewed as an improvement of \cite[Corollary 3.7]{DPO}.  

In the special case where $N = 1$, Kaliman and Zaidenberg observed that the resulting hypersurfaces could fail to be isomorphic to affine space \cite[Theorem 1]{KaZaMiyanishi}.  We conclude from Remark~\ref{rem:topologicalcontractibility} that the hypersurfaces $p_{k,l,m,f} = 0$ are not $\aone$-connected and Theorem~\ref{thm:aonecontractiblehypersurfaces} does not guarantee the hypersurface $X_{p,1}$ is $\aone$-contractible.  
\end{ex}

{\begin{footnotesize}
\raggedright
\bibliographystyle{alpha}
\bibliography{HighlyConnected}
\end{footnotesize}}

\Addresses
\end{document}